\documentclass[12pt, a4paper]{amsart}

\usepackage[utf8]{inputenc}
\usepackage[hmargin=30mm, vmargin=25mm, includefoot, twoside]{geometry}
\usepackage[bookmarksopen=true, urlcolor=magenta,urlbordercolor={1 0 0}]{hyperref}
\hypersetup{colorlinks=true,linkcolor=black,anchorcolor=blue,citecolor=magenta,filecolor=blue,urlcolor=cyan,bookmarksnumbered=true}

\usepackage[alphabetic]{amsrefs}

\usepackage{amsfonts,amssymb,verbatim}
\usepackage{latexsym}
\usepackage{mathrsfs}
\usepackage{stmaryrd}
\usepackage{xspace}
\usepackage{enumerate, paralist}
\usepackage{graphicx}
\usepackage[all,cmtip]{xy}
\usepackage[all]{xypic}

\usepackage[capitalize]{cleveref}

\usepackage[usenames,dvipsnames]{xcolor}

\usepackage{txfonts, pxfonts}

\newtheorem{cor}{Corollary}[section]

\newtheorem{pr}[cor]{Problem}
\newtheorem{te}[cor]{Theorem}
\newtheorem{p}[cor]{Proposition}
\newtheorem{qu}[cor]{Question}
\newtheorem{lemma}[cor]{Lemma}
\newtheorem{conj}[cor]{Conjecture}
\theoremstyle{definition}
\newtheorem{de}[cor]{Definition}
\theoremstyle{remark}
\newtheorem{rem}[cor]{Remark}
\newtheorem{ex}[cor]{Example}
\newtheorem{ob}[cor]{Observation}

\newcommand{\cz}{\mathbb{C}}
\newcommand{\nz}{\mathbb{N}}
\newcommand{\zz}{\mathbb{Z}}

\newcommand{\ff}{\mathbb{F}}

\newcommand{\pp}{\mathcal{P}}

\def\acts{\curvearrowright}
\newcommand{\vp}{\varphi}
\newcommand{\ve}{\varepsilon}

   {\begin{verse}%
     \small }%
   {\end{verse}}

\providecommand{\Z}{\mathbb{Z}}

\def\tilde{\widetilde}
\newcommand{\s}[1]{S_{#1}}

\begin{document}
\title{Constraint stability in permutations and action traces}

\author{Goulnara Arzhantseva}
\address{Universit\"at Wien, Fakult\"at f\"ur Mathematik\\
Oskar-Morgenstern-Platz 1, 1090 Wien, Austria.}
\email{goulnara.arzhantseva@univie.ac.at}

\author{Liviu P\u aunescu}
\address{Institute of Mathematics of the Romanian Academy, 21 Calea Grivitei Street, 010702 Bucharest, Romania}
\email{liviu.paunescu@imar.ro}
\date{}

\thanks{L.P. was supported by grant number PN-II-RU-TE-2014-4-0669 of the Romanian National Authority for Scientific Research, CNCS - UEFISCDI} 
\baselineskip=16pt

\begin{abstract}
An action trace is a function naturally associated to a probability measure preserving action of a group on a standard probability space. For countable amenable groups,  we characterise stability in permutations using action traces. We extend such a characterisation to constraint stability.  We give sufficient conditions for a group to be constraint stable. As an application, we obtain many new examples of groups stable in permutations, in particular, among free amalgamated products over a finite group. This is the first general result  (besides trivial case of free products)  which gives a wealth of non-amenable groups stable in permutations.

\end{abstract}
\maketitle

\section{Introduction}
Let $\s{n}$ be the symmetric group on the set $[n]=\{1, \ldots, n\}$ and $1_n$ denote the identity element. The \emph{normalised Hamming distance} is defined, for two permutations $p,q\in \s{n}$, by
$$d_H(p,q)=\frac1nCard\left\{i:p(i)\neq q(i)\right\}.$$

Let us consider the commutator relator
$
xyx^{-1}y^{-1}=1
$
as an equation in  $\s{n}$. A \emph{solution} of this commutator equation  is given by two permutations $p, q\in \s{n}$ which commute: $$pqp^{-1}q^{-1}=1_n.$$  If we fix the value of one of the variables in the commutator relator, that is, if we impose a \emph{constraint} to this equation, prescribing $x=a$ for a fixed  $a\in \s{n}$, then a solution of the centralizer equation $aya^{-1}y^{-1}=1$ \emph{with coefficient} $a$, is given by a permutation $q\in \s{n}$ which centralizes $a$: $$aqa^{-1}q^{-1}=1_n.$$ An \emph{almost} solution is when the above equalities to $1_n$ are relaxed to be almost equalities,  with respect to $d_H$. For example, a \emph{$\delta$-solution} of the centralizer equation  with coefficient $a$, for some $\delta>0$, is a permutation $q\in \s{n}$ such that $$d_H(aqa^{-1}q^{-1}, 1_n)<\delta.$$ 

The \emph{constraint stability in permutations} is the phenomenon when every almost solution of an equation with coefficients is near (with respect to $d_H$, uniformly independently of $n$) to a solution.   For example, the constraint stability of the centralizer equation means that every permutation which almost centralizes $a$ is $d_H$-close to a centralizing  permutation.  

A general theory of constraint metric approximations by an arbitrary approximating family endowed with a bi-invariant distance (not necessarily by permutations with $d_H$) and of constraint stability of arbitrary systems of group equations has been developed in our prior article~\cite{Ar-Pa3}.

In the present paper, we introduce the notion of \emph{action trace}.  Equipped with this tool, we extend our study of constraint stability and provide new examples of groups stable in permutations with respect to $d_H$. The following result gives a general ground for our examples, see Definition~\ref{de:clift} for the terminology.

\begin{te}[Theorem \ref{th:amalgamation}]
Let $G_1$ and $G_2$ be two countable groups with a common subgroup $H$. Suppose that $G_1$ is stable in permutations and $G_2$ is $\vp$-constraint stable, for every homomorphism $\vp\colon H\to\Pi_k\s{n_k}$. Then $G_1*_HG_2$ is stable in permutations.
\end{te}

In the process, we generalise a few classical results, our conceptual results on stability of groups from~\cite{Ar-Pa2} and results on stability of amenable groups from~\cite{BLT} (precise references are given below). The study of constraint stability initiated in~\cite{Ar-Pa3}  is more general than that of stability as considered in~\cites{Ar-Pa2, BLT}. The action traces are well-suited to this more general setting and allow to interpret the use of invariant random subgroups from~\cite{BLT} in finitary terms. The next theorem is our main technical result, see Definition~\ref{def:At} and Definition~\ref{def:caT} for the terminology.

\begin{te}[Theorem \ref{te:constraint stable}]
Let $H\leqslant G$ be countable groups, $G$ amenable and $H$ finite. Let $\vp\colon H\to\Pi_k\s{{n_k}}$ be a homomorphism. Then $G$ is $\vp$-constraint stable if and only if every $\vp$-constraint action trace is $\vp$-constraint residually finite.
\end{te}

This result is a crucial ingredient towards  our main source of new examples of groups stable in permutations:

\begin{te}[Theorem \ref{thm:am}]
 Let $G_1$ be a countable group stable in permutations and $H$ be a finite subgroup. Let $G_2$ be a countable amenable group with $Sub(G_2)$ countable, every almost normal subgroup profinitely closed, and such that $H$ is acting on $G_2$. Then $G_1*_H(G_2\rtimes H)$ is stable in permutations.
\end{te}

The paper is organised as follows. In Section \ref{sec:prel}, we fix the notation and explain conceptually some prior results. In Section~\ref{sec:action traces}, we define the action trace. Then we give a characterisation of stability in permutations for amenable groups 
using action traces, see Theorem \ref{thm:A}.  In Section~\ref{sec:constraint}, we review the notion of constraint stability and give an alternative to~\cite{Ar-Pa3} formulation, in a more group-theoretical language. Then, we prove a characterisation, analogous to Theorem \ref{thm:A}, of more general constraint stability, see Theorem \ref{te:constraint stable}. In Section~\ref{sec:examples constraint}, we give sufficient conditions for a group to be constraint stable. In Section~\ref{sec:examples stable}, we provide new examples of groups stable in permutations, obtained from our study of constraint stability via action traces.
We conclude, in Section~\ref{open}, with results on (very) flexible stability, on finite index subgroups stable in permutations, and a few open questions.

\section{Preliminaries}\label{sec:prel}
Let $\omega$ be a non-principal ultrafilter on $\nz$ and let $n_k\in\nz^*$ such that $\lim_{k\to\omega}n_k=\infty$. The metric ultraproduct of $\s{{n_k}}, k\in\nz$ with respect to the normalised Hamming distance is the \emph{universal sofic group}~\cite{El-Sz}: $$\Pi_{k\to\omega}\s{{n_k}}=\Pi_{k}\s{{n_k}}\slash\{(p_k)_{k}\in\Pi_{k}\s{{n_k}}:\lim_{k\to\omega}d_H(p_k,1_{n_k})=0\},$$
endowed with the bi-invariant metric defined by $d_{\omega}\left( \left( p_{k
}\right)_{k}  ,\left( q_{k }\right)_{k}  \right) =\lim_{k \rightarrow 
\omega}d_{H}\left( p_{k },q_{k }\right) $.
We write $1_\omega$ for the identity element of this group and denote by 
$$Q\colon \Pi_{k}\s{{n_k}}\twoheadrightarrow \Pi_{k\to\omega}\s{{n_k}}$$ the canonical projection homomorphism.

Using matrices, $\s{n}$ is identified with the group of permutation matrices.  Then $d_H(p,1_n)=1-Tr(p),$ where $Tr(p)$ is the normalised trace of the matrix $p\in \s{n}$. We define  $Tr\left((p_k)_k\right)=\lim_{k \rightarrow 
	\omega}Tr\left( p_{k }\right)$ on $\Pi_{k\to\omega}\s{n_k}.$

\begin{de}[Sofic morphism / sofic representation]\label{def:srep}
A group homomorphism $$\theta\colon G\to  \Pi_{k\to\omega}\s{{n_k}}$$ is called a \emph{sofic morphism} of $G$. A sofic morphism at the maximal distance to the identity, that is, a group homomorphism $$\theta\colon G\hookrightarrow  \Pi_{k\to\omega}\s{{n_k}}$$ with $Tr(\theta(g))=0$ for all $g\not=1_G$ in $G$, is called a \emph{sofic representation}. 
\end{de}

\begin{de}[Conjugated morphisms]\label{def:conj}
Two sofic morphisms $\theta_1,\theta_2\colon G\to\Pi_{k\to\omega}\s{n_k}$  are called \emph{conjugated} if there exist $p\in \Pi_{k\to\omega}\s{{n_k}}$ such that $\theta_1(g)=p\theta_2(g)p^{-1}$ for every $g\in G$.
\end{de}

The following result of Elek and Szab\'o is central to the theory of sofic groups. 
\begin{te}\cite{El-Sz2}*{Theorem 2}
A countable group $G$ is amenable if and only if every two sofic representations of $G$ are conjugated.
\end{te}

The next definition incorporates two results from our prior work.

\begin{de}\cite{Ar-Pa2}*{Theorem 4.2 and Theorem~7.2(i)}\label{def:lift}
	A countable group $G$ is called \emph{stable in permutations} if every sofic morphism is \emph{liftable}, i.e.\ for every homomorphism $\theta\colon G\to\Pi_{k\to\omega}\s{n_k}$ there exists a homomorphism $\vp\colon G\to\Pi_k\s{n_k}$, called a \emph{lift} of $\theta$, such that $\theta=Q\circ\vp$: 
	\[
	\xymatrix{
		& \Pi_k\s{n_k} \ar[d]^-{Q} \ar@{<--}[dl]_-{\exists\,\varphi} \\
		G \ar[r]^-{\theta} &\Pi_{k\to\omega}\s{n_k}
	}
	\]
	
	A countable group $G$ is called \emph{weakly stable in permutations} if every sofic representation $\theta\colon G\hookrightarrow  \Pi_{k\to\omega}\s{{n_k}}$ is liftable.
	
\end{de}

\begin{te}\cite{Ar-Pa2}*{Theorem 1.1}\label{thm:ws}
A countable amenable group is weakly stable in permutations if and only if it is residually finite.
\end{te}

Both \cite{El-Sz2}*{Theorem 2} and \cite{Ar-Pa2}*{Theorem 1.1} were originally stated for finitely generated groups. They hold true for countable groups as well, by the diagonal argument.

The main ingredients of our proof of \cite{Ar-Pa2}*{Theorem 1.1} were as follows. The direct implication is by two results: (1) countable amenable groups are sofic and (2)  sofic groups weakly stable in permutations are residually finite. The reverse implication is by the Elek--Szab\'o Theorem and the result that a sofic morphism, conjugated to a liftable one, is liftable. 

It turns out that the above reasoning can be generalised to stability in permutations and this was done in~\cite{BLT}, using the invariant random subgroups. In the present paper,
we generalise it further, to constraint stability in permutations.

The main difficulty is to extend the Elek--Szab\'o Theorem. This can be done by the Newman--Sohler Theorem, see~\cite{NS:pr} and~\cite{NS}*{Theorem 3.1}, a result in the setting of hyperfinite graphs. Its first appearance in the context of stability in permutations is in~\cite{BLT}*{Propostion 6.8}.  In order to extend it further to constraint stability in permutations, we introduce the notion of action trace, see Definition~\ref{def:At}. This is a finitary equivalent of invariant random subgroups that is better suited to our langage of ultrafilters and ultraproducts. It can be viewed as a generalisation of the usual trace on von Neumann algebras. Moreover, it is easily adaptable to  the setting of constraint metric approximations that we discuss in Section \ref{sec:constraint}.

\section{Action traces}\label{sec:action traces}

Let $G$ be a countable discrete group and $\left(X,\mu\right)$ be a standard probability space. Denote by $\pp_f(G)$ the set of finite subsets of $G$. Let $\alpha\colon G\to Aut\left(X,\mu\right)$ be a probability measure preserving action. We introduce the following invariant associated to the action.

\begin{de}[Trace]\label{def:tr}
The \emph{trace} of $\alpha\colon G\acts \left(X,\mu\right)$ is defined as follows: for each $A\in\pp_f(G)$,
\[Tr_\alpha(A)=\mu(\{x\in X:\alpha(g)(x)=x,\ \forall g\in A\}).\]
\end{de}

We use $Tr$ without index when the action $\alpha$ is clear from the context.

\begin{de}[Action trace]\label{def:At}
A function $Tr\colon \pp_f(G)\to[0,1]$ is called an \emph{action trace}  if there exists a probability measure preserving action $\alpha\colon G\to Aut(X,\mu)$ such that $Tr=Tr_\alpha$.
\end{de}

\subsection{Action traces of homomorphisms}

If a group $G$ admits a homomorphism to $S_n$, to the cartesian product $\Pi_kS_{n_k}$ or to  the universal sofic group $\Pi_{k\to\omega}S_{n_k}$,
 then there is a natural action trace defined by such a homomorphism, induced by the canonical action $\pi\colon S_n\acts(\{1,\ldots,n\}, \mu_n)$, where $\mu_n$ is the normalised cardinal measure.

\begin{de}[Action traces of homomorphisms]
\begin{enumerate}[(i)]
\item If $\theta\colon G\to S_n$ is a homomorphism, then we define $Tr_\theta=Tr_{\pi\circ\theta}$, where  $Tr_{\pi\circ\theta}$ is the trace of the action $\pi\circ \theta\colon G\acts(\{1,\ldots,n\}, \mu_n).$
\item If $\theta\colon G\to \Pi_kS_{n_k}$ is a homomorphism, then we define $Tr_\theta=\lim_{k\to\omega}Tr_{q_k\circ \theta}$, where $q_k\circ \theta\colon G\to S_{n_k}$  and $q_k\colon\Pi_kS_{n_k}\twoheadrightarrow S_{n_k}$ is the canonical projection on the $k$-th factor. 

Such an action trace is said to be \emph{residually finite}. 
\item
If $\theta\colon G\to \Pi_{k\to\omega}S_{n_k}$ is a sofic morphism, then we define $Tr_\theta$ to be the trace of the induced action on the Loeb measure space $G\acts (X_\omega, \mu_\omega)$, where  $X_\omega=\Pi_k X_{n_k}/\!\sim_\omega$ is the algebraic ultraproduct of $X_{n_k}=\{1, \ldots, n_k\}$ and $\mu_\omega=\lim_{k\to\omega}\mu_{n_k}$~\cite{Ar-Pa2}*{Section 2.2}. 

Such an action trace is said to be \emph{sofic}. 
\end{enumerate}
\end{de}

\begin{ob}\label{ob:notD}
For an action trace $Tr$, being residually finite, or sofic, does not depend on the sequence $\{n_k\}_k$. Indeed, if there exists a homomorphism $\theta\colon G\to \Pi_kS_{n_k}$ such that $Tr=Tr_\theta$, then there exists such a homomorphism for any other sequence $\{m_k\}_k$, provided that $\lim_{k\to\omega}m_k=\infty$. The proof is the same as our proof of  \cite{Ar-Pa2}*{Proposition 6.1}.
\end{ob}

The following result is straightforward, by definitions. 

\begin{lemma}\label{lem:Tr_lift}
Let $\widetilde \theta\colon G\to \Pi_kS_{n_k}$ be a lift of a sofic morphism $\theta\colon G\to \Pi_{k\to\omega}S_{n_k}$, then $Tr_{\widetilde\theta}=Tr_\theta$.
\end{lemma}

\begin{de}[Free action trace]
	Let $Tr_f\colon \pp_f(G)\to[0,1]$ be the action trace associated to a free action, defined by $Tr_f(\{1_G\})=1$ and $Tr_f(A)=0$, whenever $A\not=\{1_G\}$.
\end{de}

\begin{p} \label{p:free}
We have the following characterisations.
	\begin{enumerate}
		\item A sofic morphism $\theta$ is a sofic representation if and only if $Tr_\theta=Tr_f$;
		\item A group $G$ is residually finite if and only if $Tr_f$ is residually finite;\label{p:free2}
		\item A group $G$ is sofic if and only if $Tr_f$ is sofic. 
	\end{enumerate}
\end{p}
\begin{proof}
The first assertion is by definitions of sofic representation and traces. For the direct implication of the second assertion, we construct the homomorphism $\vp\colon G\to\Pi_k\s{n_k}$ by letting $G$ act on $G/N_k$, where $\{N_k\}_k$ is a decreasing chain of finite index subgroups with $\cap_kN_k=\{1_G\}$. For the reverse implication, observe that the existence of an injective homomorphism into a cartesian product of finite groups is equivalent to residual finiteness. The third item is by \cite{El-Sz}*{Theorem 1} stating that a group is sofic if and only if it admits a sofic representation.
\end{proof}

\subsection{Action traces of amenable groups}  
The next result is the action trace generalisation of the well-known fact that every amenable group is sofic.
An equivalent statement, in the setting of invariant random subgroups, is  \cite{BLT}*{Proposition 6.6}.
We provide a different proof, using ultraproducts and sofic equivalence relations. In this context, the result is a consequence of the known fact that every amenable equivalence relation is sofic. 

\begin{p}
\label{amenable-sofic}
	If $G$ is a countable amenable  group, then every action trace  is sofic.
\end{p}
\begin{proof}
	Let $Tr$ be an action trace of $G$ and  $\alpha\colon G\to Aut(X,\mu)$ be a probability measure preserving action such that $Tr=Tr_\alpha$. Let $E_\alpha$ be the orbit equivalence relation of $\alpha$ on $(X,\mu)$. Since $G$ is amenable, by the Ornstein--Weiss theorem, $E_\alpha$ is hyperfinite. It follows that $E_\alpha$ is treeable. By \cite{Pa1}*{Proposition 3.16}, $E_\alpha$ is a sofic equivalence relation (cf.~\cite{El-Li} that uses a different but, by \cite{Pa1}*{Proposition 3.22}, equivalent terminology). 
	
	Let $M(E_\alpha)$ be the tracial von Neumann algebra associated to $E_\alpha$ by the Feldman--Moore construction and $A\subseteq M(E_\alpha)$ be the corresponding Cartan pair of $E_\alpha$~\cite{Pa1}*{Section 2.3}.  By \cite{Pa1}*{Proposition 2.17}, there exists a sofic embedding $\theta\colon M(E_\alpha)\to\Pi_{k\to\omega}M_{n_k}$ into the metric ultraproduct of matrix algebras equipped with the normalised trace.  
	
	The image of $\alpha$ is included in $[E_\alpha]$, the full group of $E_\alpha$, where $[E_\alpha]=\{\vp\in Aut(X,\mu):(x,\vp(x))\in E_\alpha\ \forall x\}$. Then, using the canonical injection $\iota\colon [E_\alpha] \hookrightarrow M(E_\alpha)$~\cite{Pa1}*{Definition 2.13}, we have a map $\iota\circ\alpha\colon G\to M(E_\alpha)$. For  a finite subset $F\subseteq G$, let $c_F=\{x\in X:\alpha(g)(x)=x\ \forall g\in F\}$, and let $Q_F\in A$ be the projection on $c_F$. Then, by construction of $M(E_\alpha)$, $Tr(Q_F)=\mu(c_F)=Tr_\alpha(F)$.
	
	Let us prove that $\theta\circ\iota\circ\alpha\colon G\to\Pi_{k\to\omega}\s{n_k}$ is the required morphism.  The image is in $\Pi_{k\to\omega}\s{n_k}$, by definition of a sofic embedding~\cite{Pa1}*{Definition 2.16}. For finite $F\subseteq G$, let $P_F$ be the projection on the set of common fixed points in the Loeb measure space of $\theta\circ\iota\circ\alpha(g)$ for all $g\in F$.  We have to show that $Tr(P_F)=Tr_\alpha(F)$.
	
	We show that actually $P_F=\theta(Q_F)$. Since $\theta$ is trace preserving, $Tr(\theta(Q_F))=Tr(Q_F)=Tr_\alpha(F)$, then this concludes the proof. For every $g\in G$, $Tr(P_{\{g\}})=Tr(\theta\circ\iota\circ\alpha(g))=Tr(\iota\circ\alpha(g))=Tr(Q_{\{g\}})=Tr(\theta(Q_{\{g\}}))$. Since $\theta(Q_{\{g\}})\leqslant P_{\{g\}}$, it follows that $\theta(Q_{\{g\}})=P_{\{g\}}$. Since $\theta$ is a morphism, then $P_F=\Pi_{g\in F}P_{\{g\}}=\Pi_{g\in F}\theta(Q_{\{g\}})=\theta(\Pi_{g\in F}Q_{\{g\}})=\theta(Q_F)$.
\end{proof}

The next result is a generalisation of the Elek--Szab\'o Theorem~\cite{El-Sz2}*{Theorem 2}. It  is essentially the Newman--Sohler Theorem mentioned above, see~\cite{NS:pr} and~\cite{NS}*{Theorem 3.1}. Our proof, using action traces and action graphs, is in the arXiv version (v1) of the present article. Here,  in unison with the preceding proof and for completeness, we present a subsequent alternative proof from \cite{Ha-El}*{Theorem 5.1}, using  
action traces and sofic equivalence relations.
\begin{te}\label{thm:generalised Elek} 
Let $G$ be a countable amenable group. Let $\theta_1,\theta_2\colon G\to\Pi_{k\to\omega}S_{n_k}$ be sofic morphisms. Then, $\theta_1$ and $\theta_2$ are conjugated if and only if $Tr_{\theta_1}=Tr_{\theta_2}$. 
\end{te}
\begin{proof} We use the above notation.
For every action trace $Tr$, there exists a Bernoulli action $\beta\colon G\to Aut(X,\mu)$ such that $Tr_\beta=Tr$. For every homomorphism $\theta\colon G\to\Pi_{k\to\omega}S_{n_k}$, there exists a sofic embedding $\Phi\colon M(E_\beta)\to\Pi_{k\to\omega}M_{n_k}$ such that $\Phi|_{\iota\circ\beta}=\theta$~\cite{Ha-El}*{Theorem 1.1}. Let  $\Phi_1$ and $\Phi_2$ be such sofic embeddings associated to given $\theta_1$ and $\theta_2$. Since $G$ is amenable, then $E_\beta$ is hyperfinite, by the Ornstein--Weiss theorem. These $\Phi_1$ and $\Phi_2$ are conjugated by an element of $\Pi_{k\to\omega}\s{n_k}$ by \cite{Pa1}*{Proposition 1.20}. This element also conjugates $\theta_1$ and $\theta_2$.
\end{proof}

\begin{ex}\label{ex:Trg}
 The hypothesis  $Tr_{\theta_1}=Tr_{\theta_2}$  in Theorem~\ref{thm:generalised Elek} is on all trace numbers $Tr_{\theta}(\{g_1,\ldots,g_n\})$, for all $ g_1,\ldots, g_n\in G, n\in \nz^\ast$. 
Requiring only $Tr_{\theta_1}(\{g\})=Tr_{\theta_2}(\{g\})$ for all $g\in G$  is not sufficient to deduce the conjugacy of $\theta_1$ and $\theta_2$. Here is a counter-example, even in finite groups.

Let $G=\zz_2\times\zz_2=\langle a,b \mid a^2=b^2=(ab)^2=1\rangle$ and define $\theta_1,\theta_2\colon G\to \s{6}$ as follows:
\begin{align*}
\theta_1(a)=(12)(34)(5)(6),\ \ &\theta_1(b)=(12)(3)(4)(56),\ \  \theta_1(ab)=(1)(2)(34)(56);\\
\theta_2(a)=(12)(34)(5)(6),\ \ &\theta_2(b)=(13)(24)(5)(6),\ \ \theta_2(ab)=(14)(23)(5)(6).
\end{align*}
Then the homomorphisms $\theta_1, \theta_2$ satisfy $Tr(\theta_1(g))=Tr(\theta_2(g))=1/3$, or equivalently,  $Tr_{\theta_1}(\{g\})=Tr_{\theta_2}(\{g\})=1/3$ for all $g\not=1_G$ in $G$. However, $\theta_2$ 
has two global fixed points,
while $\theta_1$ does not have any. We deduce that $\theta_1$ and $\theta_2$ are not conjugated.

\end{ex}

 It is interesting to compare Theorem~\ref{thm:generalised Elek}  with an analogous result on \emph{hyperlinear morphisms}. It might be known to experts, although it is not in the literature.
 We formulate it in our terms, using hyperlinear analogues of Definition~\ref{def:srep} and Definition~\ref{def:conj}, where $(S_{n_k}, d_H)$ is replaced by $(U_{n_k}, d_{HS})$, the finite rank unitary group  endowed with the normalised Hilbert-Schmidt distance, defined, for two unitary matrices $u,v\in U_{n},$ by
$
d_{HS}(u,v)=\sqrt{Tr(u-v)^\ast(u-v)},
$
where $Tr$ is the normalised trace.  

\begin{te}\label{thm:hyper}
Let $G$ be a countable amenable group and $\theta_1,\theta_2\colon G\to\Pi_{k\to\omega}U_{n_k}$  be hyperlinear morphisms. Then, $\theta_1$ and $\theta_2$ are conjugated if and only if $Tr(\theta_1(g))=Tr(\theta_2(g))$ for all $g\in G$.\end{te}
\begin{proof}
We prove the non-trivial ``if'' direction. Let $\vp\colon G\to\cz$ be defined by $\vp(g)=Tr(\theta_i(g))$ with $i=1$ or $2$. Then $\vp$ is a positive defined function, invariant on conjugacy classes, i.e. a character. Let $(M,Tr)$ be the von Neumann algebra generated by the GNS representation associated with $(G,\vp)$. Since $G$ is amenable, then $M$ is hyperfinite.

The von Neumann algebra generated by $\theta_1(G)$ inside $\Pi_{k\to\omega}U_{n_k}$ is isomorphic to $(M,Tr)$. The same is true for $\theta_2(G)$. These are two embeddings of the same hyperfinite von Neumann algebra into $\Pi_{k\to\omega}U_{n_k}$. By \cite{MR2338855}*{Proposition 1}, translated into the ultraproduct language by standard arguments, these two embeddings are conjugated. 
\end{proof}

In particular, there is indeed a unitary matrix which conjugates $\theta_1$ and $\theta_2$ in Example~\ref{ex:Trg}, although we have seen that there is no such permutation matrix. 

Thus, the preceding two formulations using action trace and usual trace, respectively,  allow us to distinguish sofic and hyperlinear morphisms of amenable groups. In the hyperlinear case, we deal with the trace, i.e. a classical character. While in the sofic case, we require the action trace, which is a `character like' function associated to the action $G\acts (X_\omega, \mu_\omega)$ on the Loeb measure space. The difference in the hypothesis on these two types of traces  explains a greater  difficulty to prove the stability results in permutations versus analogous results in unitary matrices.

\subsection{Stability in permutations for amenable groups}
The following known results on stability in permutations will be generalised to the setting of constraint stability in permutations in Section~\ref{sec:constraint}.
We give new proofs, in our language of ultraproducts and action traces. 

The next result is the action trace generalisation of  \cite{Ar-Pa2}*{Theorem 4.3} that states that a sofic group stable in permutations has to be residually finite.

\begin{p}\cite[Theorem 1.3 (i)]{BLT}.\label{sofic-RF}
Let $G$ be a countable group. If $G$ is stable in permutations, then any sofic action trace is residually finite.
\end{p}
\begin{proof}
Let $Tr\colon \pp_f(G)\to[0,1]$ be a sofic action trace. Thus, there exists $\theta\colon G\to\Pi_{k\to\omega}S_{n_k}$ such that $Tr=Tr_\theta$. Since $G$ is stable in permutations, then there exists $\widetilde\theta\colon G\to\Pi_kS_{n_k}$ such that $\theta=Q\circ{\tilde\theta}$. By Lemma~\ref{lem:Tr_lift},  $Tr_\theta=Tr_{\widetilde\theta}$, and hence, $Tr$ is residually finite.
\end{proof}

The next result and its proof are the straightforward action trace generalisation of \cite{Ar-Pa2}*{Theorem 1.1}  that states that a countable amenable group is weakly stable in permutations if and only if it is residually finite. 

\begin{te}\cite{BLT}*{Theorem 1.3 (ii)}.\label{thm:A}
Let $G$ be a countable amenable group. Then $G$ is stable in permutations if and only if every action trace is residually finite.
\end{te}
\begin{proof}
If $G$ is stable in permutations, it follows by Propositions \ref{amenable-sofic} and \ref{sofic-RF} that every action trace is residually finite.

Conversely,  let  $G$ be a group such that every action trace is residually finite. Let $\theta\colon G\to\Pi_{k\to\omega}S_{n_k}$ be a sofic morphism. Since $Tr_\theta$ is residually finite and by Observation~\ref{ob:notD}, there exists $\alpha\colon G\to\Pi_kS_{n_k}$ such that $Tr_\alpha=Tr_\theta$. By Theorem \ref{thm:generalised Elek}, there exists $p\in\Pi_kS_{n_k}$ such that $Q\circ(p\alpha(g) p^{-1})=\theta(g)$ for every $g\in G$. Then $p\alpha p^{-1}$ is a lift of $\theta$, and hence, $G$ is stable in permutations.
\end{proof}

\subsection{Action traces and invariant random subgroups}  

As we alluded to above, action traces can be viewed as finitary analogues of invariant random subgroups. We explain this analogy for the reader's convenience. It is not used in our arguments.

Given a countable discrete group $G$, we denote by $2^G$ the power set of $G$ and by $Sub(G)$ the set of subgroups of $G$, endowed  with the subspace topology induced by the product topology on $2^G$. 
Since $Sub(G)$ is a closed subset of $2^G$, it is compact (as $2^G$ is, by Tychonoff's theorem). The group $G$  acts on $Sub(G)$ by conjugation. An \emph{invariant random subgroup} (briefly, IRS) is a $G$-invariant Borel probability measure on $Sub(G)$.

Let $\alpha\colon G\to Aut(X,\mu)$ be a probability measure preserving action.  Then $Stab\colon X\to Sub(G), x\mapsto stab_\alpha(x)$, where $stab_\alpha(x)$ is the stabiliser subgroup, is a $G$-equivariant function.  The pushforward measure $Stab_\ast\mu$ is an IRS that we denote by $\mu_\alpha$. If we also consider the action trace $Tr_\alpha$ associated to the action $\alpha$, see Definition~\ref{def:tr}, then we see that the two objects are linked by the formula: $Tr_\alpha(A)=\mu_\alpha(\{H\leqslant G:A\subseteq H\})$ for each $A\in\pp_f(G)$. Indeed,
\begin{align*}
	Tr_\alpha(A)&=\mu(\{ x\in X : \alpha(g)x=x, \forall g\in A\})=\mu(\{x\in X:A\subseteq stab_\alpha(x)\})\\
	&=\mu(\{Stab^{-1}(H):\forall H\leqslant G, A\subseteq H\})=\mu_\alpha(\{H\leqslant G:A\subseteq H\}).
\end{align*}

Given an arbitrary  IRS $\mu$,  by a result of Abert, Glasner and Virag~\cite{AGV}*{Proposition 13}, see also \cite{Ha-El}*{Section 3} and references therein, there exists  a probability measure preserving action $\alpha\colon G\to Aut(X,\mu)$ such that $\mu_\alpha=\mu$. Then $Tr_\alpha$ of this $\alpha$, see Definition~\ref{def:tr},  is the action trace canonically associated to $\mu$. Alternatively, we can directly define the action trace $Tr$ associated to $\mu$ using the above formula, without constructing $\alpha$. However, to satisfy Definition~\ref{def:At}, we still need to ensure that this $Tr$ is induced by a measure-preserving action.

For the reverse construction, given an action trace $Tr\colon \pp_f(G)\to[0,1]$, by Definition~\ref{def:At}, there exists (but not necessarily explicitly given) an associated probability measure preserving action
 $\alpha\colon G\to Aut(X,\mu)$. Then $\mu_\alpha=Stab_\ast\mu$ is the IRS associated to $Tr$. Alternatively, notice that an IRS is determined by its values on the fundamental sets $\{H\leqslant G:A\subseteq H, B\cap H=\emptyset\}$, for $A,B\in\pp_f(G)$. These values can be computed from $Tr$, using the inclusion-exclusion principle:
\[\mu(\{H\leqslant G:A\subseteq H, B\cap H=\emptyset\})=\sum_{V\subseteq B}(-1)^{|V|}\cdot Tr(A\cup V).\]

Thus, even if an action trace is never a measure (as the map $\pp_f(G)\ni A\mapsto \{H\leqslant G:A\subseteq H\}\subseteq Sub(G)$ is not functorial on the union of sets, and $\pp_f(G)$ is neither a $\sigma$-algebra nor an algebra),
there is a canonical transition between action traces and IRS's. However, action traces are easier to define and more tractable due to their finitary nature, especially in the von Neumann algebras like setting where they generalise the trace on type II$_1$ factors. For instance, this is evident when comparing sofic and hyperlinear morphisms in Theorems \ref{thm:generalised Elek} and \ref{thm:hyper}. Moreover, action traces admit a natural generalisation to \emph{constraint} action traces, as we show in Section \ref{sec:caT}.

\section{Constraint stability of metric approximations}\label{sec:constraint}

Let $G_1,G_2$ be two groups with a common subgroup $H$. We would like to analyse whether  $G_1*_HG_2$ is stable in permutations.  Take a sofic morphism $\theta\colon G_1*_HG_2\to\Pi_{k\to\omega}S_{n_k}$. Assume that $G_1$ is stable in permutations, so construct $\vp_1\colon G_1\to\Pi_kS_{n_k}$ such that $Q\circ\vp_1=\theta|_{G_1}$.  In order to prove stability in permutations of  $G_1*_HG_2$ it is enough to get $\vp_2\colon G_2\to\Pi_kS_{n_k}$ such that $Q\circ\vp_2=\theta|_{G_2}$ and $\vp_2|_H=\vp_1|_H$. So, we want a lift for the sofic morphism $\theta|_{G_2}$ that is prescribed on the subgroup $H$. The existence of such a lift is what we call \emph{constraint stability}, see~\cite{Ar-Pa3} and below.

\subsection{Constraint stability in permutations}

The following are instances of a general concept of a \emph{constraint lift}~\cite[Definition 2.15]{Ar-Pa3} and of a general theorem characterising \emph{constraint stability as a lifting property} of constraint morphisms~\cite[Theorem 2.16]{Ar-Pa3}.

\begin{de}[Constraint morphism / constraint lift / constraint stability]\label{de:clift}  Let $H\leqslant G$ be countable groups and $\vp\colon H\to\Pi_k\s{n_k}$, $\theta\colon G\to\Pi_{k\to\omega}\s{n_k}$ be homomorphisms. We say that:
	\begin{enumerate}[(i)]
		\item $\theta$ is \emph{$\vp$-constraint} if $\theta|_H=Q\circ\vp$; 
		\item $\theta$ is  \emph{$\vp$-constraint liftable} if there exists a homomorphism $\tilde\theta \colon G\to\Pi_k\s{n_k}$, called a \emph{$\vp$-constraint lift} of $\theta$, such that $\theta=Q\circ\tilde\theta$ and $\tilde\theta |_H=\vp$;
		
		\item $G$ is   \emph{constraint $\vp$-stable} if every $\vp$-constraint homomorphism $\theta\colon G\to\Pi_{k\to\omega}\s{n_k}$ is $\vp$-constraint liftable.\label{de:clift_iii}
	\end{enumerate}
	
	We say a \emph{liftable} homomorphism (it was termed \emph{perfect} in~\cite[Definition 4.1]{Ar-Pa2}), a \emph{lift},  and $G$ is \emph{stable in permutations},
	whenever $H=\{1_H\}$ is the trivial subgroup in the preceding definitions. 
\end{de}

In the case  $H=\{1_H\}$, we recover Definition~\ref{def:lift}.

In \cite{Ar-Pa3}, the constraint stability has been introduced in countable groups using the language of equations with coefficients and for arbitrary metric approximations (not only by permutations).  In the present paper, Definition~\ref{de:clift} uses pairs $H\leqslant G$ of countable groups and their homomorphisms instead of equations. These two viewpoints on constraint stability are easily seen to be equivalent in the setting of finitely generated groups:  fixing a finite set of generators of $H$ as coefficients leads to finitely many equations as in~\cite[Definition 2.4]{Ar-Pa3} and, conversely, generating a subgroup by given coefficients, yields a pair $H\leqslant G$ as in Definition~\ref{de:clift}.

In \cite{Ar-Pa3}, the constraint stability has been characterised as a lifting property in a theorem, the above-mentioned \cite[Theorem 2.16]{Ar-Pa3}. This characterisation is now the content of Definition~\ref{de:clift}~(\ref{de:clift_iii}).

Definition~\ref{de:clift}  extends immediately to arbitrary metric approximations and has a natural reformulation, in the spirit of \cite{Ar-Pa3}, as constraint stability of almost solutions of systems consisting of countably many equations with coefficients.

\subsection{Constraint action traces}\label{sec:caT}
Now, we transport the results of the previous section to a more general \emph{constraint} setting. That is, we fix a subgroup $H$ of $G$ and a homomorphism $\vp\colon H\to\Pi_k\s{n_k}$. Every homomorphism of $G$ to $\Pi_k\s{n_k}$ or $\Pi_{k\to\omega}\s{n_k}$ will be an extension of $\vp$ or $Q\circ\vp$, respectively.

\begin{de}[Constraint action traces]\label{def:caT}  Let $H\leqslant G$ be countable groups, $\vp\colon H\to\Pi_k\s{n_k}$ be a homomorphism and $Tr\colon \pp_f(G)\to[0,1]$ be an action trace. We say that:
\begin{enumerate}[(i)]
\item $Tr$ is \emph{$\vp$-constraint} if $Tr(A)=Tr_\vp(A)$ for each $A\in\pp_f(H)$;  
\item  $Tr$ is   \emph{$\vp$-constraint residually finite} if there exists a homomorphism $\theta\colon G\to\Pi_k\s{n_k}$ such that $Tr=Tr_\theta$ and $\theta|_H=\vp$;

\item$Tr$ is   \emph{$\vp$-constraint sofic} if there exists a sofic morphism $\theta\colon G\to\Pi_{k\to\omega}\s{n_k}$ such that $Tr=Tr_\theta$ and $\theta|_H=Q\circ\vp$.

\end{enumerate}
\end{de}

Observe that an action trace that is $\vp$-constraint sofic or $\vp$-constraint residually finite   has to be $\vp$-constraint. The next two propositions yield the converse statements, under the assumptions  on amenability or constraint stability, respectively.

\begin{p}\label{c-amenbale-sofic}
Let $H\leqslant G$ be countable amenable groups and $\vp\colon H\to\Pi_k\s{{n_k}}$ be a homomorphism. Let $Tr\colon\pp_f(G)\to[0,1]$ be a $\vp$-constraint action trace. Then $Tr$ is $\vp$-constraint sofic.
\end{p}
\begin{proof}
By Proposition \ref{amenable-sofic}, there exists $\theta\colon G\to\Pi_{k\to\omega}\s{{n_k}}$ such that $Tr=Tr_\theta$. Then, $\theta|_H$ and $Q\circ\vp$ are two sofic morphisms of $H$ with the same action trace. By Theorem \ref{thm:generalised Elek}, there exists $p\in\Pi_{k\to\omega}\s{{n_k}}$ such that $p(\theta|_H)p^{-1}=Q\circ\vp$. Then, $p\theta p^{-1}$ is the required sofic morphism of $G$. \end{proof}

\begin{p}\label{c-sofic-RF}
Let $H\leqslant G$ be countable groups and $\vp\colon H\to\Pi_k\s{{n_k}}$ be a homomorphism such that $G$ is $\vp$-constraint stable. 
Let $Tr\colon\pp_f(G)\to[0,1]$ be a $\vp$-constraint sofic action trace. Then $Tr$ is $\vp$-constraint residually finite. 
\end{p}
\begin{proof}
Since $Tr$ is  $\vp$-constraint sofic, then there exists $\theta\colon G\to\Pi_{k\to\omega}\s{{n_k}}$ such that $\theta|_H=Q\circ\vp$ and $Tr=Tr_\theta$. Since $G$ is $\vp$-constraint stable, then there exists a homomorphism $\widetilde\theta\colon G\to\Pi_k\s{{n_k}}$ such that $\theta=Q\circ\widetilde\theta$ and $\widetilde\theta|_H=\vp$. By Lemma~\ref{lem:Tr_lift}, we have $Tr_{\widetilde\theta}=Tr_\theta$. Therefore, $Tr$ is $\vp$-constraint residually finite.
\end{proof}

These two propositions immediately imply:

\begin{cor}\label{cor:cstab}
Let $H\leqslant G$ be countable amenable groups and $\vp\colon H\to\Pi_k\s{{n_k}}$ be a homomorphism such that $G$ is $\vp$-constraint stable. Then every $\vp$-constraint action trace is $\vp$-constraint residually finite.
\end{cor}

We shall prove the converse of this corollary whenever $H$ is a finite subgroup of $G$.

\begin{p}\label{prop:finite conjugated}
Let $H$ be a finite group and $\vp_1,\vp_2\colon H\to\Pi_k\s{{n_k}},$ be two conjugated homomorphisms such that $\lim_{k\to\omega}d_H(\vp_1^k(h),\vp_2^k(h))=0$ for every $h\in H$. Then, there exists $p_k\in \s{{n_k}}$ with $\lim_{k\to\omega}d_H(p_k,1_{n_k})=0$ such that $(p_k)_k\in \Pi_k\s{{n_k}}$ conjugates $\vp_1$ to $\vp_2$.
\end{p}
\begin{proof}
Let $\ve_k=\max\{d_H(\vp_1^k(h),\vp_2^k(h)):h\in H\}$. Since $H$ is finite, $\lim_{k\to\omega}\ve_k=0$. 

Let $A_k=\{i:\vp_1^k(h)(i)=\vp_2^k(h)(i), \forall h\in H\}\subseteq\{1,\ldots,n_k\}$. Then, $Card(A_k)/n_k\geqslant 1-Card(H)\cdot\ve_k$. Also, $A_k$ is invariant under $\vp_1^k$ and $\vp_2^k$. Because $\vp_1^k$ and $\vp_2^k$ are conjugated, they are conjugated also on the complement $(A_k)^c$. We construct $p_k\in S_{n_k}$ such that $p_k=1_{n_k}$ on $A_k$, and $p_k$ conjugates $\vp_1^k$ and $\vp_2^k$ on $(A_k)^c$. Then, $(p_k)_k$ conjugates $\vp_1$ to $\vp_2$, and $\lim_{k\to\omega}d_H(p_k,1_{n_k})\leqslant \lim_{k\to\omega}Card(H)\cdot\ve_k=0$.
\end{proof}

\begin{ex}
	Proposition~\ref{prop:finite conjugated} does not hold if $H$ is an arbitrary infinite group. Let $a_k,b_k\in\s{k(k-1)}$ be defined by:
	\begin{align*}
		a_k=&(1,2,\ldots,k)(k+1,\ldots, 2k)\cdots (k^2-2k+1,\ldots,k^2-k);\\
		b_k=&(1,2,\ldots,k-1)(k+1,\ldots,2k-1)\cdots (k^2-2k+1,k^2-k-1)(k,2k,\ldots,k^2-k).
	\end{align*}
	By construction, $a_k$ has $k-1$ cycles of length $k$, and $b_k$ has $k$ cycles of length $k-1$. These permutations are different only on inputs of type $mk-1$ and $mk$. Therefore, $d_H(a_k,b_k)=2(k-1)/(k^2-k)=2/k$.
	
	Let us consider $\vp_1,\vp_2\colon \zz\to\Pi_k\s{2k(k-1)}$ defined by $\vp_1(1)=(a_k\oplus b_k)_k$ and $\vp_2(1)=(b_k\oplus a_k)_k$. Clearly, $\vp_1$ is conjugated to $\vp_2$ and $d_H(\vp_1^k(1),\vp_2^k(1))=2/k$, so it tends to $0$ as $k\to\infty$. However, every $p_k\in \s{2k(k-1)}$ that conjugates $\vp_1^k(1)$ to $\vp_2^k(1)$ has the property $p_k(\{1,\ldots,k(k-1)\})=\{k^2-k+1,\ldots,2k(k-1)\}$. Thus, $d_H(p_k,1_{2k(k-1)})=1$.
	
This example shows that if we want to conjugate $a,b\in\s{n}$ by an element $p\in\s{n}$ such that $d_H(p,1_n)$ is small, it does not matter if $a(i)=b(i)$ most of the time. One needs the cycle equality, not the pointwise equality. So, the equality $a(i)=b(i)$ is useful only if $a(j)=b(j)$ for all points $j$ in the same $a$-cycle with $i$. This can be obtained if $a$ and $b$ have cycles of a fixed maximal length, like in Proposition \ref{prop:finite conjugated}.
\end{ex}

\begin{te}\label{te:constraint stable}
Let $H\leqslant G$ be countable groups, $G$ amenable and $H$ finite. Let $\vp\colon H\to\Pi_k\s{{n_k}}$ be a homomorphism. Then $G$ is $\vp$-constraint stable if and only if every $\vp$-constraint action trace is $\vp$-constraint residually finite.
\end{te}
\begin{proof}
The ``only if'' direction follows by Corollary~\ref{cor:cstab}. For the ``if'' direction, let $\theta\colon G\to\Pi_{k\to\omega}\s{{n_k}}$ be a sofic morphism such that $Q\circ\vp=\theta|_H$. Then $Tr_\theta$ is a $\vp$-constraint action trace. By hypothesis, $Tr_\theta$ is $\vp$-constraint residually finite. Then, there exists a homomorphism $\Phi\colon G\to\Pi_kS_{n_k}$ such that $\Phi|_H=\vp$ and $Tr_\Phi=Tr_\theta$. 

By Theorem \ref{thm:generalised Elek}, $\theta$ and $Q\circ\Phi$ are conjugated.  Let $p\in\Pi_k\s{{n_k}}$ be such that $Q\circ(p\cdot\Phi\cdot p^{-1})=\theta$.
So $Q\circ(p\cdot\Phi|_H\cdot p^{-1})=\theta|_H=Q\circ\vp$. By Proposition \ref{prop:finite conjugated}, applied to $p\cdot\Phi|_H\cdot p^{-1}$ and $\vp$, we get $q\in\Pi_k\s{{n_k}}$, $Q\circ q=1_\omega$ and $qp\cdot\Phi|_H\cdot p^{-1}q^{-1}=\vp$. Then $qp\cdot\Phi\cdot(qp)^{-1}$ is the required lift of $\theta$ as $(qp\cdot\Phi\cdot(qp)^{-1})|_H=\vp$ and $Q\circ(qp\cdot\Phi\cdot(qp)^{-1})=Q\circ(p\cdot\Phi\cdot p^{-1})=\theta$. Thus, $G$ is $\vp$ -constraint stable.\end{proof}

We do not know whether or not the ``if'' direction of this theorem holds for an arbitrary infinite subgroup $H$. The current proof is not sufficient because Proposition \ref{prop:finite conjugated} fails for some infinite $H$. However, it is still possible for the theorem to hold. In this scenario, one has to choose a specific $\Phi\colon G\to\Pi_kS_{n_k}$ that witnesses the constraint residually finiteness property of $Tr_\theta$, at the beginning of the proof. We consider that this scenario is unlikely. 

The notions of  constraint metric approximations and constraint stability that we have introduced in~\cite{Ar-Pa3}, give a rigorous framework to the study of arbitrary metric approximations and their stability of group-theoretical constructions such as free amalgamated products and HNN-extensions.
The next theorem illustrates this general approach in the case of stability in permutations. Our goal is to apply this theorem and thus provide new examples of groups stable in permutations.

\begin{te}\label{th:amalgamation}
Let $G_1$ and $G_2$ be countable groups with a common subgroup $H$. Suppose that $G_1$ is stable in permutations and $G_2$ is $\vp$-constraint stable, for every homomorphism $\vp\colon H\to\Pi_k\s{n_k}$. Then $G_1*_HG_2$ is stable in permutations.
\end{te}
\begin{proof}
Let $\theta\colon G_1*_HG_2\to\Pi_{k\to\omega}\s{n_k}$ be a homomorphism. Since $G_1$ is stable in permutations and $G_1$ injects into $G_1*_HG_2$, there exists a homomorphism $\psi_1\colon G_1\to\Pi_k\s{n_k}$ such that $Q\circ\psi_1=\theta|_{G_1}$. Let $\vp=\psi_1|_H$. By hypothesis, $G_2$ is $\vp$-constraint stable. It follows that there exists a homomorphism $\psi_2\colon G_2\to\Pi_k\s{n_k}$ such that $Q\circ\psi_2=\theta|_{G_2}$ and $\psi_2|_H=\vp$.

Now, $\psi_1|_H=\psi_2|_H$, so, by the universal property of free amalgamated products, we can construct the homomorphism $\psi_1*_H\psi_2\colon G_1*_HG_2\to\Pi_k\s{n_k}$. Moreover, $Q\circ(\psi_1*_H\psi_2)=\theta$, so $\theta$ is liftable.
\end{proof}

Theorem~\ref{th:amalgamation} remains true for arbitrary metric approximations (not necessarily  by permutations), under suitable variants of Definition~\ref{de:clift}. The results of Section~\ref{sec:action traces} and Section~\ref{sec:constraint} have natural analogues for arbitrary metric approximations.
In contrast, the use of Loeb measure space $(X_\omega, \mu_\omega)$ underlying the arguments of the next section makes the sofic approximations special.

\section{Homomorphism extension property}\label{sec:examples constraint}

In this section, we use Theorem \ref{te:constraint stable} in order to provide examples of groups that are $\vp$-constraint stable for every homomorphism $\vp$ of a subgroup. First we give some preliminaries. 

\subsection{Homomorphisms of finite groups}

\begin{de}[Coset multiplicity, $\varphi$]\label{def:cm}
Let $\vp\colon H\to \s{n}$ be a homomorphism and $N\leqslant H$ be a subgroup. Then the \emph{coset  multiplicity} $r(\vp,N)$ is the multiplicity of $H\curvearrowright H/N$ in $\vp$, divided by $n$.

For $\vp\colon H\to\Pi_k\s{{n_k}}$, $\vp=(\vp_k)_k$, we define $r(\vp,N)=\lim_{k\to\omega}r(\vp_k,N)$.
\end{de}

In other words, the coset multiplicity $r(\vp, N)$ is the normalised number of orbits with the same conjugacy class of stabilisers, which we denote by $[N]_H$.

The following invariant will be used to define the coset multiplicity for an arbitrary probability measure preserving action $H\acts (X, \mu)$. 

\begin{de}[Benjamini--Schramm statistics]\label{def:BS}
The  \emph{Benjamini--Schramm statistics} of a probability measure preserving  action   $\alpha\colon G\acts \left(X,\mu\right)$ are defined by:
\[S_\alpha(A,B)=\mu(\{x\in X:\alpha(g)(x)=x\ \forall g\in A;\ \alpha(g)(x)\neq x\ \forall g\in B\}),\]
where $A,B\in\pp_f(G)$.
\end{de}

The next result is added for completeness. We write $S$ and $Tr$ without index, for readability.

\begin{p}\label{inclusion-exclusion}
Given an action $\alpha\colon G\acts \left(X,\mu\right)$ as above, the associated numbers $S$ are determined by the numbers $Tr$, and vice versa.
\end{p}
\begin{proof}
This is straightforward, by  the inclusion-exclusion principle. For example, $S(\{g_1,g_2\}, \{h\})=Tr(\{g_1,g_2\})-Tr(\{g_1,g_2,h\})$ and $S(\{g\}, \{h_1,h_2,h_3\})=Tr(\{g\})-Tr(\{g,h_1\})-Tr(\{g,h_2\})-Tr(\{g,h_3\})+ Tr(\{g,h_1,h_2\})+Tr(\{g,h_2,h_3\})+Tr(\{g,h_3,h_1\})-Tr(\{g,h_1,h_2,h_3\})$. In general:
\[S_\alpha(A,B)=\sum_{V\subseteq B}(-1)^{|V|}\cdot Tr_\alpha(A\cup V),\] 
for every pair  $A,B\in\pp_f(G)$. Conversely, $Tr_\alpha(A)=S_\alpha(A,\emptyset)$.
\end{proof}

\begin{ob}\label{ob:cosetM} If $H$ is finite, then
\begin{enumerate}[(i)]
\item $\sum_{[N]_H}r(\vp,N)\cdot |H/N|=1$, where the sum is over the conjugacy classes  of subgroups $N\leqslant H$;\label{obi:cosetM}
\item $r(\vp,N)=S_\vp(N,N^c)/ [N_H(N):N]$, where $N^c$ denotes the complement $H\setminus N$, the definition of $S_\vp(N,N^c)$ uses the action
$H\acts (\{1,\ldots,n\}, \mu_n)$ induced by a given $\vp\colon H\to \s{n}$, and $N_H(N)$ denotes the normaliser of $N$ in $H$.
\end{enumerate}
\end{ob}

This yields the following natural definition of  the coset multiplicity for an arbitrary $H\acts (X, \mu)$.

\begin{de}[Coset multiplicity, $\alpha$]
	Let $\alpha\colon H\to Aut(X,\mu)$ be a probability measure preserving  action of a finite group $H$. For a subgroup $N\leqslant H$, we define the \emph{coset multiplicity} $r(\alpha,N)=S_\alpha(N,N^c)/[N_H(N):N]$.
\end{de}

\begin{ob}\label{ob:f}
	Let $H$ be a finite group, $\alpha\colon H\to Aut(X,\mu)$, and $\vp\colon H\to\Pi_k\s{{n_k}}$. Then, $r(\alpha,N)=r(\vp,N)$ for each subgroup $N\leqslant H$ if and only if $Tr_\alpha=Tr_\vp$.
\end{ob}

\begin{p}
Let $\vp,\psi\colon H\to \s{n}$ be two homomorphisms of a finite group $H$. Then $\vp$ and $\psi$ are conjugated if and only if $r(\vp,N)=r(\psi,N)$ for each subgroup $N\leqslant H$.
\end{p}
\begin{proof}
If $\vp,\psi$ are conjugated, then $S_\vp(N,N^c)=S_\psi(N,N^c)$ for each subgroup $N\leqslant H$. By Observation~\ref{ob:f}, this implies $r(\vp,N)=r(\psi,N)$ for each subgroup $N\leqslant H$. For the reverse statement, $r(\vp,N)=r(\psi,N)$ implies that the multiplicity of $H\curvearrowright H/N$ is the same in both $\vp$ and $\psi$. This allows the construction of a permutation that conjugates $\vp$ into $\psi$.
\end{proof}

\begin{de}[Homomorphism order]
Let $\vp\colon H\to\s{m}$ and $\psi\colon H\to \s{n}$ be two homomorphisms. We write $\vp\preccurlyeq\psi$ whenever $r(\vp,N)\cdot m\leqslant r(\psi,N)\cdot n$ for each  subgroup $N\leqslant H$.
\end{de}

\begin{lemma}\label{lem:aux}
	Let $\{a_i^k\}_{i,\, k\in\nz^*}$ and $\{b_k\}_{k\in\nz^*}$ be sequences of natural numbers. Assume that $\sum_i a_i^k\leqslant b_k$ for all $k\in\nz$ and that $\sum_i\lim_{k\to\omega}\frac{a_i^k}{b_k}=1$. Then $\lim_{k\to\omega}\frac{\sum_ia_i^k}{b_k}=1$.
\end{lemma}
\begin{proof}
	Let $\ve>0$ and choose $j\in\nz$ such that $\sum_{i=1}^j\lim_{k\to\omega}\frac{a_i^k}{b_k}>1-\ve$. Then:
	\[\lim_{k\to\omega}\frac{\sum_ia_i^k}{b_k}\geqslant\lim_{k\to\omega}\frac{\sum_{i=1}^ja_i^k}{b_k}=\sum_{i=1}^j\lim_{k\to\omega}\frac{a_i^k}{b_k}>1-\ve.\]
\end{proof}

\begin{p}\label{p:split}
	Let $\vp\colon H\to\Pi_k\s{n_k}$ be a homomorphism of a finite group $H$, and $\alpha\colon H\to Aut(X,\mu)$ be a probability measure preserving action with $Tr_\alpha=Tr_\vp$. Let $X=\sqcup_iX_i$, with $X_i$ invariant under $\alpha$. Then, there exists $\{\vp_i\colon H\to\Pi_k\s{n_k^i}\}_i$ such that:
	\begin{enumerate}
		\item $\vp=\oplus_i\vp_i$ and,  in particular, $\sum_in_k^i=n_k$;\label{p1:split}
		\item $\lim_{k\to\omega}\frac{n_k^i}{n_k}=\mu(X_i)$;\label{p2:split}
		\item The action trace of $\alpha$ restricted to $X_i$ is equal to $Tr_{\vp_i}$.\label{p3:split}
	\end{enumerate}
\end{p}
\begin{proof}
	To ease notation, let $\alpha_i=\alpha|_{X_i}$. Then, by definitions, the hypothesis, and Observation~\ref{ob:f}, $\sum_i \mu(X_i)r(\alpha_i,N)=r(\alpha,N)=r(\vp,N)$. Observe that $\sum_{[N]_H}r(\alpha_i,N)\cdot |H/N|=1$, cf. Observation~\ref{ob:cosetM} (\ref{obi:cosetM}). For $N\leqslant H$, if $r(\vp,N)=0$, we set $r_{i,N}^k=0$ and if $r(\vp,N)\neq 0$, we set:
	\[r_{i,N}^k=\big\lfloor r(\alpha_i,N)\cdot\mu(X_i)\cdot n_k\cdot \min\big\{\frac{r(\vp_k,N)}{r(\vp,N)},1\big\}\big\rfloor.\] Let $\vp_i^k=\oplus_{[N]_H}(H\acts H/N)\otimes Id_{r_{i,N}^k}$. So, $\vp_i^k\colon H\to\s{m_i^k}$, where $m_i^k=\sum_{[N]_H} r_{i,N}^k\cdot|H/N|$. Then $m_i^k\leqslant\sum_{[N]_H} r(\alpha_i,N)\mu(X_i)n_k\cdot|H/N|=\mu(X_i)\cdot n_k$. Let $\ve>0$ and choose a subset of natural numbers  $F_\ve\in\omega$ such that $r(\vp_k,N)/r(\vp,N)>1-\ve$ for every $k\in F_\ve$ and $N\leqslant H$, by Definition~\ref{def:cm}. Then  $m_i^k>\sum_{[N]_H} \big(r(\alpha_i,N)\mu(X_i)n_k(1-\ve)-1\big)\cdot|H/N|=\mu(X_i)\cdot n_k(1-\ve)-\sum_{[N]_H}|H/N|$. Therefore, $\lim_{k\to\omega}m_i^k/n_k=\mu(X_i)$. Let $m_k=\sum_im_i^k$. By Lemma~\ref{lem:aux}, $\lim_{k\to\omega}m_k/n_k=1$.
	
	Let $\vp_i=\Pi_k\vp_i^k$. Then, setting $n_k^i=m_i^k$, we obtain assertion (\ref{p2:split}). Below, we modify the value of $m_1^k$ (and hence of $n_k^1$) but  (\ref{p2:split}) remains true.

	Let us show that $Tr_{\vp_i}=Tr_{\alpha_i}$. For $N\leqslant H$, $r(\vp_i^k,N)=r_{i,N}^k/m_i^k$. Then $r(\vp_i^k,N)\leqslant r(\alpha_i,N)\mu(X_i)n_k/m_i^k$. As $\lim_{k\to\omega}m_i^k/n_k=\mu(X_i)$, it follows that $\lim_{k\to\omega} r(\vp_i^k,N)\leqslant r(\alpha_i,N)$. For the opposite inequality,  we again choose $F_\ve\in\omega$ such that $r(\vp_k,N)/r(\vp,N)>1-\ve$ for every $k\in F_\ve$  and $N\leqslant H$. Then $r(\vp_i^k,N)> (r(\alpha_i,N)\mu(X_i)n_k(1-\ve)-1)/m_i^k$, so $\lim_{k\to\omega} r(\vp_i^k,N)\geqslant r(\alpha_i,N)$. It follows that $r(\vp_i,N)=r(\alpha_i,N)$, and hence, $Tr_{\vp_i}=Tr_{\alpha_i}$ as required in assertion (\ref{p3:split}).
	
	We now show that $\oplus_i\vp_i^k\preccurlyeq\vp_k$. We have:
	\[r(\oplus_i\vp_i^k,N)\cdot m_k=\sum_ir_{i,N}^k\leqslant\sum_ir(\alpha_i,N)\mu(X_i)n_k\frac{r(\vp_k,N)}{r(\vp,N)}=r(\vp,N)n_k\frac{r(\vp_k,N)}{r(\vp,N)}=r(\vp_k,N)\cdot n_k.\]
	Let $\psi_k=\vp_k\ominus(\oplus_i\vp_i^k)$, where $\ominus$ denotes the subtraction of matrices. It does exist as $\oplus_i\vp_i^k\preccurlyeq\vp_k$ and it is the class of representations such that $\vp_k= (\oplus_i\vp_i^k) \oplus (\vp_k\ominus(\oplus_i\vp_i^k))$ .
	We replace $\vp_1^k$ in $\vp_1=\Pi_k\vp_1^k$ with $\vp_1^k\oplus\psi_k$. This gives assertion  (\ref{p1:split}). Since $\lim_{k\to\omega}m_k/n_k=1$, this replacement does not change $Tr_{\vp_1}$ or the value of $\lim_{k\to\omega}m_1^k/n_k$. Hence, assertions (\ref{p2:split}) and (\ref{p3:split}) remain true. 
\end{proof}

\subsection{Examples of groups with constraint residually finite action traces}

\begin{de}[Homomorphism extension property]\label{def:Ext} 
A pair of countable groups $H\leqslant G$ is said to be with \emph{extension property} if, for every $n\in\mathbb N$ and  for every homomorphism $\vp\colon H\to\s{n}$, there exists a homomorphism $\bar\vp\colon G\to \s{n}$ such that $\bar\vp|_H=\vp.$
\end{de}

Clearly, if $K\leqslant H$ and $H\leqslant G$ are with extension property, then $K\leqslant G$ is with extension property.

\begin{de}[Retract]
A subgroup $H$ in a group $G$ is a \emph{retract} of $G$ if there exists a homomorphism $\gamma\colon G\to H$ such that $\gamma|_H=id_H$. \end{de}

The next result is well known. We omit the proof as it is elementary.

\begin{lemma}\label{lem:retract} 
Let $H$ be a subgroup of a group $G$. The following are equivalent.
\begin{enumerate}[(i)]\label{lem:rte}
\item $H$ is a retract of $G$.
\item There exists $K\unlhd G$ such that $K\cap H=\{1_G\}$ and $G= K\rtimes H.$\label{lem:retract_ii}
\item For every homomorphism $\vp\colon H\to L$ to an arbitrary group $L$, there exists a homomorphism $\bar\vp\colon G\to L$ with $\bar\vp|_H=\vp.$\label{rte:iii}
\end{enumerate}
\end{lemma}

Thus, if $H$ is a retract of $G$, then $H\leqslant G$ is with extension property.

\begin{rem}\label{rem:nonR}
There are examples of pairs $H\leqslant G$ with extension property, where $H$ is not necessarily a retract. For instance, $\zz_p\leqslant S_p$, with a prime $p$, is with extension property. The cyclic subgroup is not a retract in $S_p$ as it has no normal complement, see Lemma \ref{lem:retract} (\ref{lem:retract_ii}).

Given a pair $H\leqslant G$, one can ask for an algorithm to decide whether or not it is with extension property. This  question was recently addressed  in complexity theory, in relation to list-decoding homomorphism codes~\cites{Wuu, Babai}.
\end{rem}

\begin{de}[Almost normal subgroup] A subgroup $L$ of a group $G$ is \emph{almost normal} if $L$ has only a finite number of conjugates in $G$, that is, if 
 $[G:N_G(L)]<\infty.$
\end{de}

It follows by definitions that being almost normal is preserved under taking homomorphic images and restrictions to a subgroup. It is well-known that every subgroup of a group $G$ is almost normal if and only if the quotient group by the center $G/Z(G)$  is finite~\cite{MR72137} if and only if every abelian subgroup of $G$ is almost normal~\cite{MR0106242}. 

\begin{de}[Profinitely closed]
A subgroup $L$ of a group $G$ is \emph{profinitely closed} if there is a sequence $(K_i)_{i=1}^\infty$ of finite index subgroups $K_i\leqslant G$ such that $L=\cap_{i=1}^\infty K_i.$
\end{de}

We are now ready for the main result of this section. It generalises~\cite{BLT}*{Proposition 8.1}. 

\begin{p}\label{prop:cRF}
Let $H\leqslant G$ be countable groups with extension property, $H$ be finite. Let $\vp\colon H\to\Pi_k\s{{n_k}}$ be a homomorphism. Suppose that $Sub(G)$, the set  of subgroups of $G$, is countable and that every almost normal subgroup of $G$ is profinitely closed. Then every $\vp$-constraint  action trace is $\vp$-constraint residually finite. 
\end{p}
\begin{proof}
Let $Tr$ be a $\vp$-constraint  action trace and choose $\alpha\colon G\to Aut(X,\mu)$ a measure preserving action such that $Tr_\alpha=Tr$. Since $Sub(G)$ is countable, we have $X=\sqcup_{N\in Sub(G)}Stab^{-1}(N)$ and $\sum_{N\in Sub(G)}\mu(Stab^{-1}(N))=1$, where $Stab\colon X\to Sub(G), x\mapsto stab_\alpha(x)$, and $stab_\alpha(x)$ is the stabiliser subgroup.

Let $N\in Sub(G)$ be such that $\mu(Stab^{-1}(N))>0$. Since $\alpha(g)\big(Stab^{-1}(N)\big)=Stab^{-1}(g^{-1}Ng)$, it follows that $\mu(Stab^{-1}(N))=\mu(Stab^{-1}(g^{-1}Ng))$ for any $g\in G$. Therefore, $N$ is an almost normal subgroup, and $\alpha$ is invariant on $\cup_{g\in G}Stab^{-1}(g^{-1}Ng)$. We partition the space $X$ into $\sqcup_iX_i$, each $X_i$ being 
equal to $\cup_{g\in G}Stab^{-1}(g^{-1}Ng)$, for some almost normal subgroup $N$. We use Proposition \ref{p:split}, to obtain morphisms $\vp_i\colon H\to\Pi_k\s{{n_k^i}}$. If $\alpha|_{X_i}$ is $\vp_i$-constraint residually finite, for each $i$, it follows that $\alpha$ is $\vp$-residually finite. So, we assume, without the loss of generality, that $X=\cup_{g\in G}Stab^{-1}(g^{-1}Ng)$. 

Let $M=N_G(N)=\{g\in G:gNg^{-1}=N\}$. Then $N\unlhd M$ and $[G:M]<\infty$. Let $[G:M]=j$ and choose $g_1,\ldots,g_j$ such that $G=\sqcup_ig_iM$. So, $\{g_iNg_i^{-1}:i=1,\ldots,j\}=\{gNg^{-1}:g\in G\}$. Then,
\[Tr(A)=\mu\{x:A\subseteq Stab(x)\}=\frac{Card\{i:A\subseteq g_iNg_i^{-1}\}}j=\frac{Card\{i: g_i^{-1}hg_i\in N\ \forall h\in A\}}j.\]

Since $N$ is profinitely closed and $N$ is normal in $M$, then there exists a decreasing chain of finite index normal subgroups $N_m$ of $M$ such that $\cap_mN_m=N$. We choose each $N_m$ such that whenever $g_i^{-1}hg_i\notin N$, for some $i=1,\ldots,j$ and $h\in H$, we also have $g_i^{-1}hg_i\notin N_m$.

We denote by $\psi_m$ the action of $G$ on $G/N_m$. Let $k_m=[M:N_m]$, and choose $h_1,\ldots,h_{k_m}$ such that $M=\sqcup_{r=1}^{k_m}h_rN_m$. Then $G=\sqcup_{i,r}g_ih_rN_m$ and $\psi_m(h)(g_ih_rN_m)=g_ih_rN_m$ if and only if $h\in g_iN_mg_i^{-1}$ (as $N_m$ is normal in $M$). As such:
\[Tr_{\psi_m}(A)=\frac{Card\{(i,r):h\in g_iN_mg_i^{-1}\ \forall h\in A\}}{j\cdot k_m}=\frac{Card\{i:A\subseteq g_iN_mg_i^{-1}\}}{j}.\]

It follows that $Tr_{\psi_m}\to_{m\to\infty} Tr$.  Moreover, $Tr(A)=Tr_{\psi_m}(A)$ for any $A\subseteq H$, by the requirement on the groups $N_m$. As a consequence, by Proposition \ref{inclusion-exclusion}, $S_\alpha(T,H\setminus T)=S_{\psi_m}(T,H\setminus T)$ for each subgroup $T\leqslant H$.   The action trace $Tr_\alpha$ is $\vp$-constraint. Thus, $S_\alpha(T,H\setminus T)=S_\vp(T,H\setminus T)$. As such, by Observation \ref{ob:cosetM}, $r(\vp,T)=r(\psi_m,T)$ for each subgroup $T\leqslant H$.

Let $\vp=(\vp_k)_k$, with $\vp_k\colon H\to\s{n_k}$. Fix $m\in\nz$ and define:
\[s_k=\min_{T\leqslant H,\, r(\vp,T)\neq 0}\big\lfloor\frac{r(\vp_k,T)\cdot n_k}{r(\vp,T)\cdot |G/N_m|}\big\rfloor.\]
By Definition~\ref{def:cm}, $r(\vp_k,T)\to_{k\to\omega}r(\vp,T)$. It follows that that $s_k\cdot |G/N_m|/n_k\to_{k\to\omega}1$. Also, for any $k$ in some set $F\in\omega$, $s_k\cdot r(\psi_m,T)\cdot |G/N_m|\leqslant r(\vp_k,T)\cdot n_k$ for each subgroup $T\leqslant H$. So, $\psi_m\otimes 1_{s_k}\preccurlyeq\vp_k$, and hence, we can consider $\vp_k\ominus\psi_m\otimes 1_{s_k}$, where $\ominus$ denotes the subtraction of matrices. By the extension property of $H\leqslant G$, we construct a homomorphism $\eta_k\colon G\to\s{n_k-s_k\cdot|G/N_m|}$ that extends $\vp_k\ominus\psi_m\otimes 1_{s_k}$. Then, $\theta_m^k=\psi_m\otimes 1_{s_k}\oplus\eta_k$ is a homomorphism of $G$ to $\s{n_k}$, such that $\theta_m^k|_H=\vp_k$ and $\lim_{k\to\omega}Tr_{\theta_m^k}=Tr_{\psi_m}$. We use a diagonal argument to finish the proof.

Let $G=\cup_m E_m$, where $\{E_m\}_m$ is an increasing sequence of finite subsets. Define $F_0=\nz$ and 
\[F_m=\{k>m:|Tr_{\psi_m}(A)-Tr_{\theta_m^k}(A)|<\frac1m\ \forall A\subset E_m\}\cap F_{m-1}.\]
Then $F_m\in\omega$, $F_m\subset F_{m-1}$ and $\cap F_m=0$. For every $k\in\nz$, we define $m_k=\max\{m:k\in F_m\}$, so that $k\in F_{m_k}\setminus F_{m_k+1}$. We construct $\theta\colon G\to\Pi_k\s{n_k}$, $\theta=\Pi_k\theta_{m_k}^k$. Then $\theta|_H=\Pi_k\theta_{m_k}^k|_H=\Pi_k\vp_k=\vp$. 

Let $A\in\pp_f(G)$ and $\ve>0$. Choose $m_0$, such that $1/m_0<\ve$, $A\subset E_{m_0}$ and $|Tr_{\psi_m}(A)-Tr(A)|<\ve$ for all $m\geqslant m_0$. For all $k\in F_{m_0}$, we have $m_k\geqslant m_0$. As $k\in F_{m_k}$, we get $|Tr_{\psi_{m_k}}(A)-Tr_{\theta_{m_k}^k}(A)|<1/m_k\leqslant 1/m_0<\ve$. Therefore, $|Tr_{\theta_{m_k}^k}(A)-Tr(A)|<2\ve$, so $\lim_{k\to\omega}Tr_{\theta_{m_k}^k}(A)=Tr(A)$. It follows that $Tr_\theta=Tr$.
\end{proof}

\begin{ex}
The assumption ``$H$ be finite" in Proposition \ref{prop:cRF} is optimal. Let $\zz=\langle a\rangle\leqslant\langle a,b \mid ab=ba\rangle=\zz^2$ be a group inclusion. Let $\{n_k\}_k$ be a sequence of odd prime numbers. We construct $\vp\colon \zz\to\Pi_k\s{n_k}$ such that $\vp(a)_k$ is a cycle of maximum length in $\s{n_k}$. Let $\tau$ be a $\vp$-constraint residually finite action trace. Then there exists $\theta\colon \zz^2\to\Pi_k\s{n_k}$ such that $Tr_\theta=\tau$. It follows that $\theta(b)_k$ is a power of $\theta(a)_k$. As $n_k$ is odd prime, then either $\theta(b)_k=1_{n_k}$ or $Tr(\theta(b)_k)=0$. As such $Tr_\theta(b)\in\{0,1\}$. However, one can easily construct a sofic morphism $\psi\colon \zz^2\to\Pi_{k\to\omega}\s{n_k}$, such that $\psi(a)=Q\circ\vp(a)$ and $Tr(\psi(b))=1/2$. Then, $Tr_\psi$ is a $\vp$-constraint action trace that is not $\vp$-constraint residually finite.
\end{ex}

\begin{cor}\label{cor:constraint}
Let $H\leqslant G$ be countable groups with extension property, $G$ amenable and $H$ finite.  Suppose that $Sub(G)$ is countable and that every almost normal subgroup of $G$ is profinitely closed. Then $G$ is $\vp$-constraint stable, for every homomorphism $\vp\colon H\to\Pi_k\s{n_k}$. 
\end{cor}
\begin{proof}
This follows by Proposition \ref{prop:cRF} and Theorem \ref{te:constraint stable}.
\end{proof}

\begin{te}\label{th:stable amalgamated}
Let $G_1$ and $G_2$ be countable groups with a common finite subgroup $H$.
Suppose that $G_1$ is stable in permutations, $G_2$ is amenable, $Sub(G_2)$ is countable and that every almost normal subgroup of $G_2$ is profinitely closed, and $H\leqslant G_2$ is with extension property. Then $G_1*_HG_2$ is stable in permutations.
\end{te}
\begin{proof}
This follows by Theorem~\ref{th:amalgamation} and Corollary~\ref{cor:constraint}. 
\end{proof}

\section{Examples of stable groups}\label{sec:examples stable}

We use the results in the last section to provide new examples of groups stable in permutations. The next result shows how to obtain pairs of groups satisfying the hypotheses of Corollary \ref{cor:constraint}.

\begin{p}\label{prop:sd}
Let $G$ be a group such that $Sub(G)$ is countable and that every almost normal subgroup of $G$ is profinitely closed. Let $H$ be a finite group acting on $G$. Then $G\rtimes H$ has countably many subgroups and every almost normal subgroup of $G$ is profinitely closed.
\end{p}
\begin{proof}
In order to prove that $G\rtimes H$ has countably many subgroups, one can use  \cite{Cu-Sm}*{Lemma 2.1}. We anyway have to study the structure of an arbitrary subgroup of $G\rtimes H$  for the other statement. 

Let $L\leqslant G\rtimes H$ be a subgroup. Define $L_e=L\cap G$. Let $\vp\colon G\rtimes H\to H, (g,h)\mapsto h$ be the canonical projection homomorphism induced by the structure of the semidirect product. We define $H_0=\vp(L)$. Choose $g_h\in L$ such that $\vp(g_h)=h$ for each $h\in H_0$. It is easy to see that $L=\cup_{h\in H_0}L_eg_h$. This shows that $Sub(G\rtimes H)$ is countable.

Assume now that $L$ is almost normal in $G\rtimes H$. Then, using the definition, we see that $L_e$ is almost normal in $G$. By hypothesis, there exists finite index subgroups $K_i$ of $G$ such that $L_e=\cap_iK_i$. We replace each $K_i$ with $\cap_{h\in H_0}g_hK_ig_h^{-1}$. Then, $K_i$ are still finite index subgroups in $G$ (since $H_0$ is finite)
and $L_e=\cap_iK_i$. Moreover, $gK_ig^{-1}=K_i$ for each $g\in L$. As such, the subgroup generated in $G\rtimes H$ by $K_i$ and $L$ is $K_iL$.

We use these subgroups to prove that $L$ is profinitely closed. Clearly, $K_iL$ are finite index subgroups of $G\rtimes H$. Let $g\in\cap_iK_iL$. Then, for each $i$, there exists $k_i\in K_i$ and $h_i\in H_0$ such that $g=k_ig_{h_i}$. Now, $\vp(g)=\vp(k_ig_{h_i})=h_i$, so $h_i$ is independent of  $i$, and $g=k_ig_h$ for some $h\in H$. Then $gg_h^{-1}\in K_i$ for each $i$, so $gg_h^{-1}\in L_e$. It follows that $g\in L$, so $\cap_i K_iL=L$, and hence, $L$ is profinitely closed.
\end{proof}

The class of groups with countably many subgroups is closed under taking subgroups and quotients but, in general, not under extensions, nor even direct products. For example, if $p$ is a prime, the Pr\"ufer $p$-group $C_{p^\infty}$ has $Sub(C_{p^\infty})$ countable, but its direct square 
$C_{p^\infty}\times C_{p^\infty}$ has $2^{\aleph_0}$ subgroups~\cite{Cu-Sm}.

According to Lemma \ref{lem:retract}, the pair $H\leqslant G\rtimes H$ is always with extension property. As such, under the hypothesis of Proposition~\ref{prop:sd}, the pair $H\leqslant G\rtimes H$ satisfies all the assumptions of Corollary~\ref{cor:constraint}. By also using  Proposition \ref{prop:sd} and Theorem \ref{th:amalgamation}, we obtain the following general result.

\begin{te}\label{thm:am}
 Let $G_1$ be a countable group stable in permutations and $H$ be a finite subgroup. Let $G_2$ be a countable amenable group with $Sub(G_2)$ countable, every almost normal subgroup profinitely closed, and such that $H$ is acting on $G_2$. Then $G_1*_H(G_2\rtimes H)$ is stable in permutations.
\end{te}

Here are some concrete examples of groups stable in permutations by Theorem \ref{thm:am}.

\begin{ex}[Virtually free examples]

 The special linear group  $SL_2(\zz)\cong \zz_4*_{\zz_2}(\zz_3\times\zz_2)$ is stable in permutations as $\zz_2\leqslant \zz_3\times\zz_2$ is with extension property and other hypothesis of Theorem~\ref{thm:am} are also satisfied.

Given arbitrary groups $G_1, G_2$ and $Q$, the semidirect product of $G_1\ast_H G_2$ by $Q$ is isomorphic to the free product of 
$G_1\rtimes Q$ and $G_2\rtimes Q$ amalgamated over $H\rtimes Q$:
\begin{equation}\tag{$\rtimes$}\label{eq:sD}
(G_1\ast_H G_2)\rtimes Q\cong (G_1\rtimes Q)\ast_{H\rtimes Q}(G_2\rtimes Q).
\end{equation}
In particular, for the general linear group: $GL_2(\zz)\cong SL_2(\zz)\rtimes \zz_2\cong(\zz_4\rtimes \zz_2)\ast_{\zz_2\rtimes\zz_2}(\zz_6\rtimes\zz_2)$. It is
stable in permutations. Indeed, by Gasch\"utz' complement theorem~\cite{compl}*{Satz 1 on p. 99}, since the normal subgroup $\zz_3\trianglelefteq (\zz_6\rtimes\zz_2)$ has a complement in $\zz_6\cong \zz_3\times \zz_2$, then it has a complement in $\zz_6\rtimes\zz_2$. It is clear that such a complement $(\zz_6\rtimes\zz_2)/ \zz_3$ is isomorphic to $\zz_2\rtimes\zz_2$. Therefore, $\zz_2\rtimes\zz_2$ is a retract of  $\zz_6\rtimes\zz_2$, and hence,  $\zz_2\rtimes\zz_2\leqslant \zz_6\rtimes\zz_2$ is with extension property. Other hypothesis of Theorem \ref{thm:am} are clearly satisfied. 

Both $SL_2(\zz)$ and $GL_2(\zz)$ are virtually free groups, and hence, they are stable in permutations also by a different proof from~\cite{VFree}. 
\end{ex}

\begin{ex}[Non-virtually free examples]
By varying the groups involved in the free amalgamated product from Theorem  \ref{thm:am}  or in the above semidirect product construction (\ref{eq:sD}), we obtain many non-amenable groups stable in permutations, which are not virtually free. 

For instance, $GL_2(\zz)*_{H}(BS(1,n)\rtimes H)$ is not virtually free and it is stable in permutations by Theorem~\ref{thm:am}. Indeed, the Baumslag-Solitar group $BS(1,n)$ satisfies the hypothesis of Theorem~\ref{thm:am} by ~\cite{BLT}*{Proof of Corollary~8.4} and $BS(1,n)\rtimes H$, where $H$ is a finite subgroup of $GL_2(\zz)$, satisfies the hypothesis of Theorem~\ref{thm:am}, by Proposition~\ref{prop:sd}. 

Gasch\"utz' type results and its generalisations~\cite{Grec} yield many pairs $H\rtimes Q\leqslant G_2\rtimes Q$ with extension property so that, also by Proposition~\ref{prop:sd}, Theorem~\ref{thm:am} applies to the above semidirect product construction~(\ref{eq:sD}), where $G_1$ and $G_2$ are as in Theorem~\ref{thm:am}.

\end{ex} 
\begin{ex}[Around just-infinite branch groups]
Let $\Gamma$ be the first Grigorchuk group or the Gupta-Sidki $p$-group. Then, $\Gamma$ is stable in permutations~\cite{Gr}*{Theorem~6.6}.
Therefore, $\Gamma*_{H}(G_2\rtimes H)$ is stable in permutations, where $H$ is a finite subgroup of $\Gamma$ and $G_2$ is an arbitrary group satisfying the hypothesis of Theorem \ref{thm:am}.

\end{ex}

\section{Further results and questions}\label{open}

\subsection{(Very) flexible stability}
There are natural dimension related relaxations of stability, called \emph{flexible stability} and \emph{very flexible stability}:  an almost solution  in  $S_n$ is required to be close, in a suitable  sense, to a solution in $S_N$, for $N$ not necessarily equal to $n$~\cite{BL}. It is straightforward to adapt our concepts and results to such a setting. For instance, Theorem \ref{thm:am} has  the following analogue.

\begin{te}\label{thm:flex}
 Let $G_1$ be a countable group flexibly (respectively, very flexibly) stable in permutations and $H$ be a finite subgroup. Let $G_2$ be a countable amenable group with $Sub(G_2)$ countable, every almost normal subgroup profinitely closed, and such that $H$ is acting on $G_2$. Then $G_1*_H(G_2\rtimes H)$ is flexibly (respectively, very flexibly)  stable in permutations.
\end{te}

This gives new examples of flexibly (respectively, very flexibly) stable groups. 

\subsection{Finite index subgroups stable in permutations}
The homomorphism extension property can be relaxed to having extensions to a finite index subgroup of $G$ containing $H$ (instead of $G$ itself).

\begin{de}[Local extension property, cf.  \cite{Reid}*{Definition 1.1.} ]\label{def:loc} 
A pair of countable groups $H\leqslant G$ has \emph{local extension property} if, for every $n\in\mathbb N$ and  for every homomorphism $\vp\colon H\to\s{n}$, there exists a finite index subgroup $K\leqslant G$, with $H\leqslant K$, and a  homomorphism $\vp^\circ\colon K\to \s{n}$ such that $\vp^\circ|_H=\vp.$
\end{de}

\begin{te}\cite{Reid}*{Theorem 1.2.}\label{thm:reid}
If $G$ is subgroup separable (or, in other terms, LERF) and $H$ is a finitely generated subgroup of $G$, then the pair $H\leqslant G$ has local extension property.
\end{te}

\begin{lemma}\label{lem:local}
Let $H\leqslant G$ be countable groups, $H$ be a finite group and $G$ be an LERF group. Then there exists a finite index subgroup $K\leqslant G$ containing $H$ such that $H\leqslant K$ is with extension property.  
\end{lemma}
\begin{proof}
Let $A=\{\vp_i\colon H\to\s{n_i}: i\in I\}$ be the finite collection of transitive actions of $H$ onto finite sets. If for some group $K$, containing $H$, these homomorphisms can be lifted to $K$, then $H\leqslant K$ is with extension property. Indeed, any homomorphism $H\to\s{n}$ is a direct sum of homomorphisms in $A$.

By Theorem~\ref{thm:reid}, for each $i\in I$, let $K_i\leqslant G$ be the finite index subgroup containing $H$ such that there exists $\vp^\circ_i\colon K_i\to\s{n_i}$ that extends $\vp_i$. Let $K=\cap_{i\in I}K_i$. Since $I$ is finite, $K$ is a finite index subgroup in $G$. Moreover, any $\vp_i$, $i\in I$ can be extended to $K$. It follows that $H\leqslant K$ is with extension property.
\end{proof}

Lemma~\ref{lem:local} yields a general finite index variant of Theorem \ref{th:stable amalgamated}.

\begin{te}\label{thm:findex}
Let $G_1$ and $G_2$ be countable groups with a common finite subgroup $H$.
Suppose that $G_1$ is stable in permutations, $G_2$ is amenable and  LERF, $Sub(G_2)$ is countable, and that every almost normal subgroup of $G_2$ is profinitely closed.
Then  $G_1*_HG_2$ has a finite index subgroup that is stable in permutations.
\end{te}
\begin{proof}
By Lemma~\ref{lem:local}, we find a finite index subgroup $K\leqslant G_2$ containing $H$ such that $H\leqslant K$ is with extension property. Then, $G_1*_HK$ has finite index in $G_1*_HG_2$ and, by Theorem \ref{th:stable amalgamated}, $G_1*_HK$ is stable in permutations.
\end{proof}

\begin{ex}
The assumption ``$H$ is finite" in Theorem \ref{thm:findex} is optimal. Indeed, $(\Z \times \Z)\ast_\Z(\Z \times \Z)=\mathbb F_2 \times \Z$ satisfies all the other assumptions (with $H=\Z$) but every finite index subgroup of $\mathbb F_2 \times \Z$ is not stable in permutations~\cite{MR4134896}.
\end{ex}

\begin{rem}\label{rem:notERF}
The assumption on almost normal subgroups in Theorem \ref{thm:findex} is not redundant. Indeed, there exist amenable LERF groups with a not profinitely closed normal subgroup: the wreath product $A\wr \Z$, where $A$ is a finitely generated abelian group, is such an example~\cite{MR2295543}*{Proposition~3.19}.
\end{rem}

\subsection{Open questions}
In proving results of Section~\ref{sec:examples constraint}, we require the homomorphism extension property of $H\leqslant G$.
By Lemma~\ref{lem:retract}, every semidirect product $G=K\rtimes H$ yields a pair $H\leqslant G$ with extension property. In Remark~\ref{rem:nonR}, we have given an example of a pair $H\leqslant G$ with extension property, where $H$ is not a retract. How far can we go from the semidirect products?

\begin{pr}
Let $G= K\bowtie H$ be the Zappa-Sz\'ep product of two groups. Characterise the pairs $H\leqslant G$ with extension property.
\end{pr}

Stability in permutations is not preserved under arbitrary amalgamated free product  or semidirect product constructions~\cite{BL}, and not even under the direct product with $\zz$~\cite{MR4134896}. In contrast, our results give many examples of non-amenable amalgamated free products  and semidirect products  which are stable  in permutations. The following basic question is still open.

\begin{qu}
Let $G$ be a countable group stable in permutations. Let $H$ be a finite group. Is $G\rtimes H$ stable in permutations?
\end{qu}

Together with the first Grigorchuk group and the Gupta-Sidki $p$-group, Grigorchuk's groups $G_\omega$, with $\omega$ in a certain uncountable subset of $\{0, 1, 2\}^{\zz_+}$, are stable in permutations~\cite{Gr}. All these uncountably many groups are amenable but not elementary amenable. They are finitely generated but not finitely presented. 

\begin{qu}
Does there exist a finitely presented amenable but not elementary amenable group stable in permutations?
\end{qu}

We expect a positive answer. A natural candidate is the finitely presented Grigorchuk group $\Gamma_\sigma=\langle \Gamma, t \mid t^{-1}\Gamma t=\sigma(\Gamma)\rangle$, where $\Gamma$ is the first Grigorchuk group and $\sigma$ is Lys\"enok's endomorphism of $\Gamma$~\cite{MR1616436}. However, although $\Gamma$ is residually finite, $\Gamma_\sigma$ is not~\cite{MR1868545}. 
Therefore, since $\Gamma_\sigma$ is amenable, then $\Gamma_\sigma$ is not stable in permutations, by~\cite[Theorem 4.3]{Ar-Pa2} (also by \cite[Theorem 2]{Gl-Ri}, using stability in permutations for presentations of groups, together with another result from~\cite{Ar-Pa2}, showing that stability is a group property, i.e., it is independent of the choice of the presentation).

The free amalgamated product from the next question is a building block of the famous Higman group~\cite{MR38348}:
$$
H\cong \left(BS(1,2)\ast_\zz BS(1,2)\right)\ast_{\ff_2} (BS(1,2)\ast_\zz BS(1,2)),
$$
an infinite group all of whose finite quotients are trivial.
It follows from  \cite[Theorem 4.3]{Ar-Pa2} that if $H$ is stable in permutations, then $H$ is  not sofic. 

In detail, let us consider the Baumslag-Solitar group, $BS(1,n)=\langle x_i,t_i  \mid t_i^{-1}x_it_i=x_i^n\rangle, i=1,2,3,4$.
Then we form three types of the free amalgamated products over an infinite cyclic group:
\begin{eqnarray*}
H(t_1,t_2)=\langle x_1, t_1 \mid t_1^{-1}x_1t_1=x_1^n\rangle\ast_{\langle t_1\rangle =\langle t_2\rangle}\langle x_2, t_2 \mid t_2^{-1}x_2t_2 =x_2^n\rangle,\\
H(x_1,x_2)=\langle x_1, t_1 \mid t_1^{-1}x_1t_1=x_1^n\rangle\ast_{\langle x_1\rangle =\langle x_2\rangle}\langle x_2, t_2 \mid t_2^{-1}x_2t_2 =x_2^n\rangle,\\
H(x_1,t_2)=\langle x_1, t_1 \mid t_1^{-1}x_1t_1=x_1^n\rangle\ast_{\langle x_1\rangle =\langle t_2\rangle}\langle x_2, t_2 \mid t_2^{-1}x_2t_2 =x_2^n\rangle.
\end{eqnarray*}
Being the free amalgamated products of sofic (even solvable) groups over an amenable (even cyclic) group, all these groups are sofic~\cites{El-Sz2,Pa1}. 

The group $H(t_1,t_2)$ is residually finite, since the amalgamation is along the retract~\cite{MR306329} or, by \cite[Theorem 1]{Malcev} or by a direct argument, because $H(t_1,t_2)=\mathbb F_2\rtimes \zz$, a semidirect product of a finitely generated residually finite group by a residually finite group.

 The group $H(x_1,x_2)$ is not Hopfian~\cite{MR38347}: an endomorphism $x_1\mapsto x_1^n, t_1\mapsto t_1, t_2\mapsto t_2$ is surjective but not injective. By a theorem of Mal'cev~\cite{MR0003420}, every finitely generated residually finite group is Hopfian. Therefore, $H(x_1,x_2)$ is not residually finite. It follows from  \cite[Theorem 4.3]{Ar-Pa2} that $H(x_1,x_2)$ is not stable in permutations. 
 
 Finally, for $n=2$, the group $H(x_1,t_2)$ is the above mentioned building block of the Higman group:  
$$
H=H(x_1,t_2)\ast_{\langle t_1,\, x_2\rangle = \langle t_3,\, x_4 \rangle}H(x_3, t_4).
$$

\begin{qu} Let $n\geqslant 2$.
Is $H(x_1,t_2)=BS(1,n)\ast_\zz BS(1,n)$ stable in permutations?
\end{qu}

Since the arXiv version (v1) of the present article appeared on arXiv, we were informed that standard Bass-Serre theory arguments show that $H(x_1,t_2)$ is not residually finite~\cite{Yves}. Then, as above, 
by \cite[Theorem 4.3]{Ar-Pa2}, $H(x_1,t_2)$ is not stable in permutations. Thus, the answer to the preceding question is negative.

\begin{conj}
The Higman group $H$ is not stable in permutations.
\end{conj}


\begin{bibdiv}
\begin{biblist}
	

\bib{AGV}{article}{
	author={Abert, M.},
	author={Glasner, Y.},
	author={Virag, B.},
	title={Kesten’s theorem for invariant random subgroups},
	journal={Duke Math. J.},
	volume={163},
	date={2014},
	number={3},
	pages={465--488},
}

	
	
\bib{Ar-Pa2}{article}{
   author={Arzhantseva, G.},
   author={P\u{a}unescu, L.},
   title={Almost commuting permutations are near commuting permutations},
   journal={J. Funct. Anal.},
   volume={269},
   date={2015},
   number={3},
   pages={745--757},
}



\bib{Ar-Pa3}{article}{
   author={Arzhantseva, G.},
   author={P\u{a}unescu, L.},
   title={Constraint metric approximations and equations in groups},
   journal={J. Algebra},
   volume={516},
   date={2018},
   pages={329--351},
}

\bib{Babai}{article}{
   author={Babai, L.},
   author={Black, T. J. F.},
   author={Wuu, A.},
   title={List-decoding homomorphism codes with arbitrary codomains},
   conference={
      title={Approximation, randomization, and combinatorial optimization.
      Algorithms and techniques},
   },
   book={
      series={LIPIcs. Leibniz Int. Proc. Inform.},
      volume={116},
      publisher={Schloss Dagstuhl. Leibniz-Zent. Inform., Wadern},
   },
   date={2018},
   pages={Art. No. 29, 18},
}



\bib{BL}{article}{
   author={Becker, O.},
   author={Lubotzky, A.},
   title={Group stability and Property (T)},
   journal={J. Funct. Anal.},
   volume={278},
   date={2020},
   number={1},
   pages={108298, 20},
}


\bib{BLT}{article}{
   author={Becker, O.},
   author={Lubotzky, A.},
   author={Thom, A.},
   title={Stability and invariant random subgroups},
   journal={Duke Math. J.},
   volume={168},
   date={2019},
   number={12},
   pages={2207--2234},
}

\bib{MR306329}{article}{
   author={Boler, J.},
   author={Evans, B.},
   title={The free product of residually finite groups amalgamated along
   retracts is residually finite},
   journal={Proc. Amer. Math. Soc.},
   volume={37},
   date={1973},
   pages={50--52},
}







\bib{MR2295543}{article}{
   author={de Cornulier, Y.},
   title={Finitely presented wreath products and double coset
   decompositions},
   journal={Geom. Dedicata},
   volume={122},
   date={2006},
   pages={89--108},
}


\bib{Yves}{article}{
   author={de Cornulier, Y.},
   journal={Personal communication},
   date={2023},
}




\bib{Cu-Sm}{article}{
   author={Cutolo, G.},
   author={Smith, H.},
   title={Groups with countably many subgroups},
   journal={J. Algebra},
   volume={448},
   date={2016},
   pages={399--417},
}










\bib{El}{article}{
   author={Elek, G.},
   title={Finite graphs and amenability},
   journal={J. Funct. Anal.},
   volume={258},
   date={2012},
   number={5},
   pages={1692--1708},
   issn={0022-1236},
}




\bib{El-Li}{article}{
   author={Elek, G.},
   author={Lippner, G.},
   title={Sofic equivalence relations},
   journal={J. Funct. Anal.},
   volume={263},
   date={2012},
   number={9},
   pages={2593--2614},
}


\bib{El-Sz}{article}{
   author={Elek, G.},
   author={Szab{\'o}, E.},
   title={Hyperlinearity, essentially free actions and $L^2$-invariants.
   The sofic property},
   journal={Math. Ann.},
   volume={332},
   date={2005},
   number={2},
   pages={421--441},
}



\bib{El-Sz2}{article}{
   author={Elek, G.},
   author={Szab{\'o}, E.},
   title={Sofic representations of amenable groups},
   journal={Proc. Amer. Math. Soc.},
   volume={139},
   date={2011},
   number={12},
   pages={4285--4291},
}





\bib{MR0106242}{article}{
   author={Eremin, I. I.},
   title={Groups with finite classes of conjugate abelian subgroups},
   journal={Mat. Sb. (N.S.)},
   volume={47 (89)},
   date={1959},
   pages={45--54},
}



\bib{compl}{article}{
   author={Gasch\"{u}tz, W.},
   title={Zur Erweiterungstheorie der endlichen Gruppen},
   journal={J. Reine Angew. Math.},
   volume={190},
   date={1952},
   pages={93--107},
}


\bib{Gl-Ri}{article}{
   author={Glebsky, L.},
   author={Rivera, L. M.},
   title={Almost solutions of equations in permutations},
   journal={Taiwanese J. Math.},
   volume={13},
   date={2009},
   number={2A},
   pages={493--500},
}




\bib{MR1616436}{article}{
   author={Grigorchuk, R. I.},
   title={An example of a finitely presented amenable group that does not
   belong to the class EG},
   journal={Mat. Sb.},
   volume={189},
   date={1998},
   number={1},
   pages={79--100},
}



\bib{MR2338855}{article}{
   author={Hadwin, D.},
   author={Shen, J.},
   title={Free orbit dimension of finite von Neumann algebras},
   journal={J. Funct. Anal.},
   volume={249},
   date={2007},
   number={1},
   pages={75--91},
}



\bib{Ha-El}{article}{
   author={Hayes, B.},
   author={Elayavalli, S.},
   title={Approximate homomorphisms and sofic approximations of orbit
   equivalence relations},
   journal={Ergodic Theory Dynam. Systems},
   volume={44},
   date={2024},
   number={12},
   pages={3455--3480},
}






\bib{MR38347}{article}{
   author={Higman, G.},
   title={A finitely related group with an isomorphic proper factor group},
   journal={J. London Math. Soc.},
   volume={26},
   date={1951},
   pages={59--61},
   label={Hig51a},
}

\bib{MR38348}{article}{
   author={Higman, G.},
   title={A finitely generated infinite simple group},
   journal={J. London Math. Soc.},
   volume={26},
   date={1951},
   pages={61--64},
    label={Hig51b},
}







\bib{MR4134896}{article}{
   author={Ioana, A.},
   title={Stability for product groups and property ($\tau$)},
   journal={J. Funct. Anal.},
   volume={279},
   date={2020},
   number={9},
   pages={108729, 32},
}


\bib{VFree}{article}{
   author={Lazarovich, N.},
   author={Levit, A.},
   title={Virtually free groups are stable in permutations},
   journal={Groups Geom. Dyn.},
   volume={17},
   date={2023},
   number={4},
   pages={1417--1434},
}







\bib{Reid}{article}{
   author={Long, D. D.},
   author={Reid, A. W.},
   title={Subgroup separability and virtual retractions of groups},
   journal={Topology},
   volume={47},
   date={2008},
   number={3},
   pages={137--159},
}





\bib{MR0003420}{article}{
   author={Malcev, A.},
   title={On isomorphic matrix representations of infinite groups},
   journal={Rec. Math. [Mat. Sbornik] N.S.},
   volume={8 (50)},
   date={1940},
   pages={405--422},
}




\bib{Malcev}{article}{
   author={Mal{\cprime}cev, A},
   title={On homomorphisms onto finite groups},
   journal={Ivanov. Gos. Ped. Inst. U\v c. Zap. Fiz.-Mat. Fak.},
   volume={18},
   date={1956},
   pages={49--60},
}


\bib{MR72137}{article}{
   author={Neumann, B. H.},
   title={Groups with finite classes of conjugate subgroups},
   journal={Math. Z.},
   volume={63},
   date={1955},
   pages={76--96},
}




\bib{NS:pr}{article}{
   author={Newman, I.},
   author={Sohler, Ch.},
   title={Every property of hyperfinite graphs is testable [extended
   abstract]},
   conference={
      title={STOC'11---Proceedings of the 43rd ACM Symposium on Theory of
      Computing},
   },
   book={
      publisher={ACM, New York},
   },
   date={2011},
}


\bib{NS}{article}{
   author={Newman, I.},
   author={Sohler, Ch.},
   title={Every property of hyperfinite graphs is testable},
   journal={SIAM J. Comput.},
   volume={42},
   date={2013},
   number={3},
   pages={1095--1112},
}




\bib{Pa1}{article}{
   author={P{\u{a}}unescu, L.},
   title={On sofic actions and equivalence relations},
   journal={J. Funct. Anal.},
   volume={261},
   date={2011},
   pages={2461--2485},
}



\bib{Grec}{article}{
   author={Sambale, B.},
   title={On the converse of Gasch\"utz' complement theorem},
   journal={J. Group Theory},
   volume={26},
   date={2023},
   number={5},
   pages={931--949},
}







\bib{MR1868545}{article}{
   author={Sapir, M.},
   author={Wise, D. T.},
   title={Ascending HNN extensions of residually finite groups can be
   non-Hopfian and can have very few finite quotients},
   journal={J. Pure Appl. Algebra},
   volume={166},
   date={2002},
   number={1-2},
   pages={191--202},
}




\bib{Schramm}{article}{
   author={Schramm, O.},
   title={Hyperfinite graph limits},
   journal={Electron. Res. Announc. Math. Sci.},
   volume={15},
   date={2008},
   pages={17--23},
}





\bib{Wuu}{article}{
   author={Wuu, A.},
   title={Homomorphism Extension},
   date={2018},
   note={arXiv:1802.08656},
}

\bib{Gr}{article}{
   author={Zheng, T.},
   title={On rigid stabilizers and invariant random subgroups of groups of homeomorphisms},
   date={2019},
   note={arXiv:1901.04428},
}


\end{biblist}
\end{bibdiv}
\end{document}